\documentclass[11pt]{amsart}
\usepackage{amsfonts}
\usepackage{amsmath}
\usepackage{amssymb}
\usepackage{graphicx}

\usepackage[bindingoffset=0.2in,%
left=1.0in,right=1.0in,top=1.25in,bottom=1.375in,%
footskip=.25in]{geometry}

\usepackage{amsthm}
\usepackage{latexsym}
\usepackage{fancyhdr}
\usepackage{amscd}
\usepackage{lscape}
\usepackage{tikz}
\usepackage{float}
\usepackage{mathtools}
\usepackage{lipsum}
\usepackage{bm}
\usepackage{subfigure}
\usepackage{booktabs}
\usepackage{grffile}
\usepackage{array}
\usepackage{caption}
\usepackage{longtable}
\usepackage{diagbox}
\usepackage[toc,page, title, titletoc]{appendix}
\usepackage[normalem]{ulem} 

\newcolumntype{C}[1]{>{\centering\arraybackslash}p{#1}}
\definecolor{navy}{HTML}{2F729C} 
\usepackage[hyperfootnotes=false, colorlinks, linkcolor={blue}, citecolor={magenta}, filecolor={blue}, urlcolor={blue}, plainpages=false, pdfpagelabels]{hyperref}
\newcommand{\fp}{\mathfrak{p}}
\newcommand{\fq}{\mathfrak{q}}
\newcommand{\OK}{\mathcal{O}_K}

\newcommand{\nup}{\nu_{\mathfrak{p}}}
\newcommand{\Q}{{\mathbb Q}}

\newcommand{\fN}{\mathfrak{N}}
\newcommand{\minD}{\mathfrak{D}^{\text{min}}}

\newcommand{\Co}{\mathcal{C}_{p, 1}(t, d)}
\newcommand{\Ct}{\mathcal{C}_{p, 2}(t, d)}
\newcommand{\Ci}{\mathcal{C}_{p, i}(t, d)}

\newcommand{\Coo}{\mathcal{C}_{p, 1}(t_0, d_0)}
\newcommand{\Cto}{\mathcal{C}_{p, 2}(t_0, d_0)}

\newtheorem{theorem}{Theorem}[section]
\newtheorem{lemma}[theorem]{Lemma}

\newtheorem{proposition}[theorem]{Proposition}
\newtheorem{corollary}[theorem]{Corollary}
\theoremstyle{definition}
\newtheorem{definition}[theorem]{Definition}
\newtheorem{example}[theorem]{Example}

\theoremstyle{remark}
\newtheorem{remark}[theorem]{Remark}

\numberwithin{equation}{section}

\begin{document}
	
	\title[Prime isogenous discriminant ideal twins]{Prime isogenous discriminant ideal twins}

	\author[A. Barrios]{Alexander J. Barrios}
	\address{Department of Mathematics, University of St. Thomas, St. Paul, Minnesota, USA}
	\email{abarrios@stthomas.edu}
	
	\author[M. Brucal-Hallare]{Maila Brucal-Hallare}
	\address{Department of Mathematics, US Air Force Academy, Colorado Springs, Colorado, USA}
	\email{mbh2739@gmail.com}
	
	\author[A. Deines]{Alyson Deines}
	\address{Center for Communications Research, San Diego, California, USA}
	\email{aly.deines@gmail.com}
	
	\author[P. Harris]{Piper Harris}
	\email{drpiperh@liberatemath.org}
	
	\author[M. Roy]{Manami Roy}
	\address{Department of Mathematics, Lafayette College, Easton, Pennsylvania, USA}
	\email{royma@lafayette.edu}



	\begin{abstract}
		Let $E_{1}$ and $E_{2}$ be elliptic curves defined over a number field $K$. We say that $E_{1}$ and $E_{2}$ are discriminant ideal twins if they are not $K$-isomorphic and have the same minimal discriminant ideal and conductor. Such curves are said to be discriminant twins if, for each prime~$\mathfrak{p}$ of $K$, there are $\mathfrak{p}$-minimal models for $E_{1}$ and $E_{2}$ whose discriminants are equal. This article explicitly classifies all prime-isogenous discriminant (ideal) twins over $\mathbb{Q}$. We obtain this classification as a consequence of our main results, which constructively gives all $p$-isogenous discriminant ideal twins over number fields where $p\in\left\{  2,3,5,7,13\right\}  $, i.e., where~$X_0(p)$ has genus~$0$. In particular, we find that up to twist, there are finitely many $p$-isogenous discriminant ideal twins if and only if $K$ is~$\mathbb{Q}$ or an imaginary quadratic field. In the latter case, we provide instructions for finding the finitely many pairs of $j$-invariants that result in $p$-isogenous discriminant ideal twins. We prove our results by considering the local data of parameterized $p$-isogenous elliptic curves.
	\end{abstract}
	\maketitle
	\setcounter{tocdepth}{1}
	\tableofcontents
	
	\section{Introduction}
	Let $E_1$ and $E_2$ be isogenous elliptic curves over a number field $K$. As $E_1$ and $E_2$ are isogenous, they have the same conductor. In this paper, we examine when they can also have the same minimal discriminant ideal. We call such elliptic curves \emph{discriminant ideal twins}. In this paper we classify $p$-isogenous discriminant ideal twins over number fields where $p \in \{2, 3, 5, 7, 13\}$, i.e., where $X_0(p)$ has genus 0 and thus the isogeny class has a parameterization given by \cite{Bariso}.
	The necessary SageMath code for the examples and results in the paper can be found on \cite{GitHubDIT}.
	This work is motivated by Deines \cite{Deines2018}, which studied the case of discriminant twins for isogenous semistable elliptic curves over $\mathbb{Q}$, that is, semistable  elliptic curves that have the same conductor and the same minimal discriminant.
	
	Our original motivation comes from \cite{DeinesThesis}, where Deines records two algorithms for computing the degree of parameterization of a semistable elliptic curve by a Shimura curve. The first, more complete, algorithm uses an algorithm by Voight and Willis \cite{voight2014computing} that returns the $j$-invariant of the optimal quotient. The optimal quotient is the curve in the isogeny class such that the map from the Shimura curve to the elliptic curve has a connected kernel. With the optimal quotient in hand, the method of Ribet and Takahashi \cite{Ribet1997ParametrizationsOE} (generalized by Deines \cite{DeinesThesis} to the setting of totally real number fields) then computes the degree of this parameterization. Alternatively, if the map on component groups is surjective (which is conjectured by Takahashi in certain cases; see \cite{papikian2016optimal} for more exposition), then one does not need the extra machinery of first computing the optimal quotient to compute the Shimura degree and can compute the degree of the parameterization directly.  However, finding the optimal quotient now depends on the uniqueness of the minimal discriminant ideals.

	In the local setting of elliptic curves with a prime of multiplicative reduction, Deines \cite{Deines2018} used the Tate curve parameterization to show the following result:
	\begin{theorem}[Deines \cite{Deines2018}]
		Isogeny classes of size two with at least one prime of multiplicative reduction cannot have discriminant ideal twins.
	\end{theorem}
	
	Moreover, Deines classified all semistable isogenous discriminant twins defined over $\mathbb{Q}$:
	
	\begin{theorem}[Deines \cite{Deines2018}]\label{Deines2018}
		Over $\mathbb{Q}$, there are only finitely many semistable isogenous discriminant twins. Given by LMFDB labels \cite{lmfdb}, they are $11a.1, 11a.3$ (25-isogenous), $17a.1, 17a.4$ (4-isogenous), $19a.1, 19a.3$, and $37b.1, 37b.3$ (both 9-isogenous).
	\end{theorem}
	
	We note that quadratic twists of the isogenous discriminant twins in Theorem~\ref{Deines2018} result in discriminant twins that are not semistable. Consequently, we obtain infinite families of isogenous discriminant twins but only finitely many different $j$-invariants.
	
	In this article, we generalize the results in \cite{Deines2018}. Namely, we remove the restriction that the elliptic curves have semistable reduction by considering all $p$-isogenous discriminant ideal twins over number fields, where $p \in \{2, 3, 5, 7, 13\}$.  We also extend the notion of discriminant twins to number fields via  local criteria (see Definition~\ref{def:disctwins}).
	
	Recall that the classical modular curve $X_0(p)$ has genus $0$ at a prime $p$ if and only if $p \in \{2, 3, 5, 7, 13\}$ \cite{MazurRatIsogPrimeDeg}. 
	In particular, we have that for these primes, the Fricke parameterizations give the $j$-invariants of $p$-isogenous elliptic curves \cite[Chapter~2.8]{MR3838339}, \cite[Table~3]{MR3084348}. 
	Utilizing the Fricke parameterizations, Cremona, Watkins, and Tsukazaki \cite{CremWat,Tsukazaki} gave the kernel polynomial of an elliptic curve that admits a $p$-isogeny.
	In \cite{Bariso}, Barrios extended these works for those $n\ge2$ such that $X_0(n)$ has genus $0$ by introducing $56$ parameterized families of elliptic curves $\mathcal{C}_{n,i}(t,d)$ defined over $\mathbb{Q}(t,d)$ with the property that for a fixed $n$, the collection $\{ \mathcal{C}_{n,i}(t,d)\}_i$ consists of non-isomorphic isogenous elliptic curves over $\mathbb{Q}(t,d)$. We note that the notation is chosen so that $\mathcal{C}_{n,i}(t,d)$ is the quadratic twist of $\mathcal{C}_{n,i}(t,1)$ by $d$. In the case when $n=p$ is prime, we have that $i \in \{1,2\}$ and $\Ci:y^2=x^3+d^2 A_{p,i}(t)x+d^3 B_{p,i}(t)$ with $A_{p,i}, B_{p,i}$ are given in Table~\ref{ta:curves} (see Theorem~\ref{parafamilies}). The parameterized elliptic curves $\Ci$ have the property that if $E_1$ and $E_2$ are $p$-isogenous elliptic curves such that their $j$-invariants are not both identically $0$ or $1728$, then there are $t_0, d_0 \in K$ such that $E_i \cong \mathcal{C}_{p,i}(t_0,d_0)$ \cite[Theorem~1]{Bariso}. With this notation, we state our main result below, which is Theorem~\ref{ThmClassification}.
	
	\begin{theorem}
		\label{mainthm} 
		Let $K$ be a number field with ring of integers $\OK$ and let $p\in\left\{ 3,5,7,13\right\}$. 
		Let $\mathcal{C}_{p,1}(t_{0},d_{0})$ and $\mathcal{C}_{p,2}(t_{0},d_{0})$ be the parameterized elliptic curves over $K$ from Theorem~\ref{parafamilies}. Suppose further that their $j$-invariants are not equal.
		Then $\mathcal{C}_{p,1}(t_{0},d_{0})$ and $\mathcal{C}_{p,2}(t_{0},d_{0})$ are discriminant ideal twins if and only if for each prime $\fp$ of $K$,
		\begin{equation}\label{Thm1_3eqn}
			\nup(t_0) = \frac{12k}{p-1} \quad \text{ for }\quad 0 \leq k \leq \nup(p).
		\end{equation}
		Furthermore, $\mathcal{C}_{p,1}(t_{0},d_{0})$ and $\mathcal{C}_{p,2}(t_{0},d_{0})$ are $p$-isogenous discriminant twins if and only if
		$t_{0}\in\mathcal{O}_{K}^{\frac{12}{p-1}}$.
	\end{theorem}
	
	Note that $k$ depends on $\nup(p)$, so the principal ideal generated by $t_0$ can only be divisible by primes dividing $p$ in $K$. Thus, the number of units in $\OK$ determines the number of discriminant (ideal) twins in $K$.
	When $p=2$, we note that the conditions on $t_0$ in Theorem~\ref{mainthm} are necessary, but not sufficient (see Theorem~\ref{PropMinDisc} and Example~\ref{ex:p2}). As a consequence of Theorem~\ref{mainthm}, we see that for $p$-isogenies with $p=3,5,7,13$, discriminant (ideal) twins are preserved under quadratic twist. This is not true for $2$-isogenies, as illustrated by the example below.
	
	\begin{example}
		Let $K=\mathbb{Q}(\sqrt[3]{2})$ and let $t_0=16$. Then $K$ has class number one, and we let $\mathfrak{p}$ denote the prime ideal generated by $\sqrt[3]{2}$. We note that $t_0$ satisfies the assumptions of Theorem~\ref{mainthm} for $p=2$ since $\nu_{\mathfrak{p}}(t_0)=12$. Furthermore, for each $d_0 \in K^\times$, $\mathcal{C}_{2,i}(16,d_0)$ is a quadratic twist of $\mathcal{C}_{2,i}(16,1)$ for $i=1,2$. Now consider the following elliptic curves over $K:$%
		\begin{align*}
			E_{1}  & =\mathcal{C}_{2,1}\!\left(  16,1\right)  :y^{2}=x^{3}%
			-587520x+171417600,\\
			E_{2}  & =\mathcal{C}_{2,2}\!\left(  16,1\right)  :y^{2}=x^{3}%
			-1105920x-176947200,\\
			F_{1}  & =\mathcal{C}_{2,1}\!\left(  16,-2\right)  :y^{2}=x^{3}%
			-2350080x-1371340800,\\
			F_{2}  & =\mathcal{C}_{2,2}\!\left(  16,-2\right)  :y^{2}=x^{3}%
			-4423680x+1415577600.
		\end{align*}

		In particular, $E_{1}$ (resp. $F_{1}$) is $2$-isogenous to $E_{2}$ (resp.
		$F_{2}$) and $F_{i}$ is the quadratic twist by $-2$ of $E_{i}$. Since $K$ is a
		class number one number field, we can consider the minimal discriminants of
		these elliptic curves, which are listed below:
		\begin{align*}
			\Delta_{E_{1}}^{\text{min}}  & =\Delta_{E_{2}}^{\text{min}}=4416972000\sqrt[3]{4}+5565036000\sqrt[3]%
			{2}+7011506000,\\
			\Delta_{F_{1}}^{\text{min}}  & =282686208000\sqrt[3]{4}+356162304000\sqrt[3]{2}%
			+448736384000,\\
			\Delta_{F_{2}}^{\text{min}}  & =17667888000\sqrt[3]{4}+22260144000\sqrt[3]{2}+28046024000.
		\end{align*}
		From this we see that $E_{1}$ and $E_{2}$ are discriminant twins, but $F_{1}$
		and $F_{2}$ are not discriminant twins. In fact, they are not discriminant
		ideal twins since $\nu_{\mathfrak{p}}(\Delta_{F_{1}}^{\text{min}})=30$ and $\nu
		_{\mathfrak{p}}(\Delta_{F_{2}}^{\text{min}})=18$. This shows that $2$-isogenous discriminant
		twins are not preserved under quadratic twist.
		
	\end{example}
	
	Theorem \ref{mainthm} is proven in Section~\ref{sec:pfmainthm}. The techniques of our proof rely on a study of the local data of the parameterized isogenous families of elliptic curves $\mathcal{C}_{p,i}(t,d)$ and on the work of Dokchitser and Dokchitser \cite{DokchitserDokchitser2015}.
	Their work describes how the valuations of the minimal discriminants of two prime isogenous elliptic curves are related (see Theorem~\ref{DokchitserSquared}).
	As a consequence of loc. cit., we show that if $E_1$ and $E_2$ are discriminant ideal twins, then they have everywhere potentially good reduction (see Lemma~\ref{lemmapotgoodred}).

	By \cite{Bariso}, Theorem \ref{mainthm} informs us how to find all
	$p$-isogenous discriminant (ideal) twins $E_{1}$ and $E_{2}$ over a given
	number field where $p\in\left\{  2,3,5,7,13\right\}  $ and the $j$-invariants
	of $E_{1}$ and $E_{2}$ are not both identically $0$ or $1728$. In Section
	\ref{Sec:special-jinv}, we consider the case where $E_{1}$ and $E_{2}$ are
	$p$-isogenous elliptic curves over a number field $K$ with $j(E_{1}%
	)=j(E_{2})$.
	In particular, we prove the following
	result for singular $j$-invariants (see Theorem~\ref{thmforj-_1728}):

	\begin{theorem}\label{mainthm2}
		Let $K$ be a number field and let $E_1$ and $E_2$ be $p$-isogenous elliptic curves defined over $K$ for $p \in \{2,3,5,7,13 \}$. 
		Let $\zeta_3 = \frac{-1 + \sqrt{-3}}{2}$ be a third root of unity and $i = \sqrt{-1}$ a fourth root of unity.
		Suppose further that $j=j(E_{1} )=j(E_{2})\in\left\{  0,1728\right\}  $.
		If $E_{1}$ and~$E_{2}$ are discriminant ideal twins, then one of the following holds:
		
		\begin{enumerate}
			\item[($a$)] $(p,j)=(3,0)$, $\zeta_{3}\not \in K$, and the ideal $3\mathcal{O}_{K}$ is a square. 
			
			\item[($b$)] $(p,j)=(2,1728)$, $i\not \in K$, and the ideal $2\mathcal{O}_{K}$ is a square.
		\end{enumerate}
		Conversely,
		\begin{enumerate}
			\item[($c$)] if $(p,j)=(3,0)$, $\zeta_{3}\not \in K$, and the ideal $3\mathcal{O}_{K}$ is a square, then $E_1$ and $E_2$ are discriminant ideal twins. Further, if $3$ is a square in $K$, then $E_1$ and $E_2$ are discriminant twins. 
			\item[($d$)] if $(p,j)=(2,1728)$, then there are no discriminant twins. 
		\end{enumerate}
	\end{theorem}
	
	To prove Theorem~\ref{mainthm2}, we first show that $p$-isogenous discriminant ideal
	twins do not exist when $p\in\left\{  5,7,13\right\}  $ (see Proposition
	\ref{Proponjs}). We then finish proving Theorem~\ref{mainthm2} by studying $p$-isogenous
	parameterized families of elliptic curves corresponding to $p\in\left\{
	2,3\right\}  $ (see Lemma~\ref{j0or1728}).
	Continuing in Section \ref{Sec:special-jinv}, Theorem~\ref{Thm:equaljinv} then considers the $p$-isogenous discriminant twins defined over a number field, $K$, now with $j(E_{1})=j(E_{2}) \not \in \{0,1728\}$.
	Consequently, Theorems \ref{mainthm}, \ref{mainthm2}, and \ref{Thm:equaljinv} classify all $p$-isogenous discriminant (ideal) twins over number fields for $p\in\left\{  2,3,5,7,13\right\}  $. 
	We also show that the converse to Theorem~\ref{mainthm2} $(b)$ is not true in general (see Example~\ref{failureat2specialj}).
	
	While the focus of this article is on prime-isogenous discriminant (ideal) twins, a similar analysis can be done in the case of $n$-isogenies for which $X_0(n)$ has genus $0$. Parameterizations for these isogenous elliptic curves are also found in \cite{Bariso}. The case of composite isogeny degrees are not as straightforward and will be considered in the sequel to this article.

	\subsection{Organization of the paper}
	In Section \ref{sec:prelim}, we start by recalling some preliminary
	definitions and results about elliptic curves. More specifically, in
	Section~\ref{sec:disctwindef}, we define discriminant (ideal) twins, and in
	Section~\ref{sec:isofamily}, we introduce the isogeny parameterizations of
	Barrios \cite{Bariso}.
	
	The isogeny parameterizations in loc. cit. concern $p$-isogenous elliptic
	curves whose $j$-invariants are not both identically $0$ or $1728$.
	Consequently, in Section~\ref{Sec:special-jinv}, we first consider the case of
	$p$-isogenous discriminant (ideal) twins whose $j$-invariants are both
	identical.
	We then proceed to Section~\ref{sec:pfmainthm},
	which considers the case of $p$-isogenous discriminant (ideal) twins whose
	$j$-invariants are not both identically $0$ or $1728$. We conclude the section
	by showing that if a number field $K$ has infinite unit group, then there are,
	up to twist, infinitely many $p$-isogenous discriminant (ideal) twins for
	$p\in\left\{  3,5,7,13\right\}  $ (see Theorem~\ref{InflyDITs}). As a
	consequence, a number field $K$ admits, up to twist, finitely many discriminant
	ideal twins if $K$ is either $\mathbb{Q}$ or an imaginary quadratic field.
	
	We conclude our article with Section~\ref{sec:Examples}, where we use the
	results of our paper to explicitly classify, up to twist, all prime isogenous
	discriminant (ideal) twins over $\mathbb{Q}$ (see Proposition \ref{classinQ}). Next, we consider $p$-isogenous
	discriminant ideal twins over an imaginary quadratic field $K$ for
	$p\in\left\{  2,3,5,7,13\right\}  $. If the unit group of the ring of integers
	of $K$ is $\left\{  \pm1\right\}  $, Proposition \ref{CorQuaClass} provides
	instructions for finding, up to twist, all $p$-isogenous discriminant ideal
	twins. As a consequence, we classify the $j$-invariants corresponding to
	$p$-isogenous discriminant ideal twins over $\mathbb{Q}(\sqrt{-33})$ (see Corollary \ref{corQ33}). If instead, the unit group of the
	ring of integers of $K$ is not $\left\{  \pm1\right\}  $, i.e., $K$ is either
	$\mathbb{Q}(i)$ or $\mathbb{Q}(\zeta_{3})$, then we explicitly classify, up to twist, all $p$-isogenous
	discriminant (ideal) twins (see Proposition \ref{PropImKu}).

	\subsection{Acknowledgments.} 
	This paper is an outgrowth of the workshop Rethinking Number Theory: 2020, which was organized by Heidi Goodson, Christelle Vincent, and Mckenzie West. The authors extend their sincere thanks to the workshop organizers, without which this paper would not have been written. We also thank the anonymous referees for reading our manuscript very carefully and providing many valuable comments and suggestions. The AMS Simons Travel Grant program partially supported MR during this work.

	\section{Preliminaries}\label{sec:prelim}
	In what follows, we recall some facts about elliptic curves. See Silverman~\cite{Silverman2009} for details. An elliptic curve, $E$, is defined over a field, $K$, if $E$ is given by a Weierstrass model
	\begin{equation}\label{eq:1}
		E:y^2 + a_1 xy + a_3 y = x^3 + a_2 x^2 + a_4x + a_6     
	\end{equation}
	with $a_i \in K.$  From the Weierstrass coefficients, one defines the
	quantities%
	$$
	\begin{array}
		[c]{l}%
		c_{4}=a_{1}^{4}+8a_{1}^{2}a_{2}-24a_{1}a_{3}+16a_{2}^{2}-48a_{4},\\
		c_{6}=-\left(  a_{1}^{2}+4a_{2}\right)  ^{3}+36\left(  a_{1}^{2}%
		+4a_{2}\right)  \left(  2a_{4}+a_{1}a_{3}\right)  -216\left(  a_{3}^{2}%
		+4a_{6}\right)  ,\\
		\Delta=\frac{c_{4}^{3}-c_{6}^{2}}{1728},\qquad\text{and}\qquad j=\frac
		{c_{4}^{3}}{\Delta}.
	\end{array}
	$$ 
	Now, suppose that $E$ is as given above. The elliptic curve 
	$$E^\prime : y^2 + a_1' xy + a_3' y = x^3 + a_2' x^2 + a_4'x + a_6' $$
	is $K$-isomorphic to $E$ if there is an admissible change of variables $\tau: E \rightarrow E'$ defined by $(x,y) \mapsto (u^2 x + r,u^3 y + su^2x + w)$, where $u, r, s, w \in K$, with $u \neq 0$. We write $\tau = [u, r, s, w].$ Consequently,
	\[
	j^{\prime} = j, \qquad \Delta^{\prime}=u^{-12}\Delta,\qquad c_{4}^{\prime}%
	=u^{-4}c_{4},\qquad c_{6}^{\prime}=u^{-6}c_{6}.
	\]
	Note that the $j$-invariant does not depend on the choice of model, but the discriminant $\Delta$ does. We will frequently use that an admissible change of variables only changes the discriminant by twelfth powers of $u$. If any of $u, r, s, w \not \in K$, then $E$ and $E^\prime$ are isomorphic over an algebraic closure of $K$. When this occurs, i.e., when $\tau$ is defined over an extension of $K$, $E^\prime$ is said to be a \textit{twist} of $E$ if $E'$ is also defined over $K$. In this case, we say that $\tau$ is a twist.
	
	\subsection{Minimal discriminants}
	\label{sec:mindiscs}
	
	To start, let $K$ be a number field or local field. 
	An elliptic curve $E:y^2 + a_1 xy + a_3 y = x^3 + a_2 x^2 + a_4x + a_6$ defined over $K$ is given by an 
	\textit{integral Weierstrass model} if each $a_i$ is in the ring of integers of $K$ 
	for $i \in \{1, 2, 3, 4, 6\}$.
	For the rest of this section, let $K$ be a number field and denote its ring of integers by $\OK$. 
	Let $\fp$ be a prime ideal of $K$ and $\nup$ be the normalized valuation of the completion $K_{\fp}$ of $K$ at~$\fp$. 
	Further, let $R_{\fp}$ be the ring of integers of $K_{\fp}$ and let $\pi$ be a uniformizer for $R_{\fp}$. Via the inclusion $\iota:K \hookrightarrow K_{\fp}$ we will view curves defined over $K$ as defined over $K_{\fp}$. 
	As our goals are ultimately to say something global, we will choose $\pi \in R_{\fp}$ so that pulling back via the inclusion $\pi \in \mathcal{O}_K$. 
	In particular, we will choose $\pi$ so that $\nup(\pi) = 1$ and $\nu_{\mathfrak{q}}(\pi) \geq 0$ for all other primes $\mathfrak{q}$ of $\mathcal{O}_K.$ 
	In fact, if $\fp$ is not principal and $\pi \in \fp$ satisfies $\nup(\pi)=1$, 
	then there is at least one distinct prime $\mathfrak{q}$ with $\nu_\mathfrak{q}(\pi)\geq1$. 

	With the notation set, we discuss some local properties of elliptic curves. Let $E$ be an elliptic curve defined over $K_{\fp}$ with Weierstrass equation as in \eqref{eq:1}. 
	The transformation $\tau = [u, 0, 0, 0]$ on $E$ gives a $K_{\fp}$-isomorphic elliptic curve whose Weierstrass coefficients are $u^{-i}a_i$. 
	Thus, choosing $u$ to be a sufficiently large power of $\pi$ we can find a model for $E$ such that all coefficients are in $R_{\fp}$, i.e., we can always find a $\fp$-integral Weierstrass model.
	Note that given our choice for~$\pi$, if $E$ is defined over $K$, this $\fp$-integral model is globally integral.
	The integral Weierstrass model, by construction, has an integral discriminant, $\nup(\Delta) \geq 0.$ As $\nup$ is discrete, there will be an integral model such that $\nup(\Delta) \geq 0$ is minimal. This model will not be unique, but this minimal valuation of the discriminant will be.
	
	\begin{definition}
		A Weierstrass model for $E$ over $K_{\fp}$ is called a \textit{$\fp$-minimal (Weierstrass) model} if $\nup(\Delta)$ is minimized subject to the constraint that the model is an integral Weierstrass model. Any $\Delta$ such that $\nup(\Delta)$ is minimal is called a \textit{$\fp$-minimal discriminant} of $E$, and we call $\nup(\Delta)$ the \textit{valuation of the minimal discriminant} of $E$ at $\fp$.
	\end{definition}
	
	Further, it follows from \cite[Proposition~1.3]{Silverman2009} that every elliptic curve $E$ defined over $K_{\fp}$ has a $\fp$-minimal model and the valuation of the discriminant of any $\fp$-minimal model is unique up to a change of coordinates $\tau = [u, r, s, w]$ with $u \in R_{\fp}^*$ and $r, s, w \in R_{\fp}$.

	Next, we examine the global properties of elliptic curves in relation to the valuation of the minimal discriminant of $E$ at $\fp$. Let $E$ now be defined over a number field $K$. 
	As above, for each prime $\fp$ of $K$, we can embed $\iota: K \hookrightarrow K_{\fp}$.
	Using this embedding, we can consider the base change of $E$ to $K_{\fp}$. 
	Thus for each prime $\fp$ we can choose a $\fp$-minimal model, $E_{\fp}$ defined over $K$, given by an integral Weierstrass model, that is $K$-isomorphic to $E$ with $\fp$-minimal discriminant $\Delta_{\fp}$. This, together with \cite[Proposition
	VII.1.3]{Silverman2009} leads us to the following lemma, which we will use in proving our main results:
	
	\begin{lemma}\label{newSilverman}
		Let $K$ be a number field and let $E$ be an elliptic curve defined over $K$ with an
		integral Weierstrass model. For a prime $\mathfrak{p}$ of $K$, there is an
		elliptic curve $E^{\prime}$ that is $K$-isomorphic to $E$ such that
		$E^{\prime}$ is a $\mathfrak{p}$-minimal model. In addition, the isomorphism
		$\tau:E\rightarrow E^{\prime}$ is of the form $\tau=\left[  u,r,s,w\right]  $,
		where $u,r,s,w\in \mathcal{O}_K$ with $u\neq0$.
	\end{lemma}
	
	Putting together the information from all the $\fp$-minimal models we can construct the \textit{minimal discriminant ideal.}
	
	\begin{definition}
		Let $E$ be an elliptic curve defined over a number field $K$. The \textit{minimal discriminant ideal} of $E$, denoted by $\minD$, is the (integral) ideal of $K$ given by 
		$$\minD = \prod_{\fp} \fp^{\nup(\Delta_{\fp})}.$$
	\end{definition}
	
	We note that over a number field, it is not always possible to find a single Weierstrass equation that is simultaneously minimal for every prime $\fp$ of $K$. Following Silverman's exposition in \cite[Section VIII.8]{Silverman2009}, we examine the obstruction to finding a single Weierstrass equation that is simultaneously minimal for every prime $\fp.$ Start with a Weierstrass equation for $E$ as in \eqref{eq:1} with discriminant $\Delta.$ For each prime $\fp$ we can find a change of variables $\tau_{\fp} = [u_{\fp}, r_{\fp}, s_{\fp}, w_{\fp}]$ such that $\tau_{\fp}(E)$ is minimal at $\fp$ with discriminant $\Delta_{\fp}.$ Then the two discriminants are related by $\Delta = u_{\fp}^{12} \Delta_{\fp}$. Let us define the ideal $\mathfrak{a}_{E} = \prod_{\fp} \fp^{-\nup(u_{\fp})}.$ Then we can write the relationship between the minimal discriminant ideal $\minD$ and $\Delta$:
	$$\minD = \left(\Delta\right) \mathfrak{a}_{E}^{12}.$$

	Silverman notes in  \cite[Lemma 8.1]{Silverman2009} that the ideal class of $\mathfrak{a}_{E}$ in the ideal class group of $K$ is independent of $\Delta$, leading to the following definition:  the \textit{Weierstrass class} of $E$, denoted by $[\mathfrak{a}_{E}]$, is the ideal class in $K$ corresponding to any ideal $\mathfrak{a}_{E}$ as above. Thus, the obstruction to finding a single Weierstrass equation that is simultaneously minimal for every prime $\fp$ is the Weierstrass class of $E$. 
	
	\begin{definition}
		A \textit{global minimal model} for $E$ over $K$ is an integral Weierstrass equation 
		$$y^2 + a_1 xy + a_3 y = x^3 + a_2 x^2 + a_4 x + a_6$$
		such that the discriminant $\Delta$ of the equation satisfies $\minD = \left( \Delta \right).$
	\end{definition}
	
	Such a global Weierstrass model for $E$ exists if and only if the Weierstrass class of $E$ is trivial, i.e., $[\mathfrak{a}_{E}] = [\left( 1 \right)]$ \cite[Proposition 8.2, Section VIII]{Silverman2009}. Thus, if $K$ has class number one, then every elliptic curve over $K$ has a global minimal model.
	
	\subsection{Conductors and Local data of elliptic curves}
	\label{sec:localdata}
	The discriminant and conductor of an elliptic curve are closely related and measure its arithmetic complexity. Before defining the conductor, we describe the possible reduction types of an elliptic curve.
	An elliptic curve $E$ defined over $K_\mathfrak{p}$ has the following possible \textit{reduction types}: \textit{good reduction} when $\nup\!\left(  \Delta_{\fp}\right)=0$, \textit{multiplicative reduction} when $\nup\!\left(  \Delta_{\fp}\right)>0$ and $\nup\!\left(  c_{4}\right)=0$, and \textit{additive reduction} if $\nup\!\left(  \Delta_{\fp}\right) >0$ and $\nup\!\left(c_{4}\right)>0$. We further say that $E$ has \textit{potentially good (resp.\ multiplicative) reduction} if it has good (resp.\ multiplicative) reduction over a finite extension of $K_{\fp}$.
	Now suppose that $E$ is an elliptic curve $E$ defined over a number field $K$. Then, we say that $E$ has good, multiplicative, additive,   potentially good, or potentially multiplicative reduction at a prime $\fp$ if the base change of $E$ over $K_{\fp}$ has the same reduction, respectively. 
	
	\begin{definition}
		The \textit{conductor}, $\fN$, of an elliptic curve $E$ defined over a number field $K$ is an ideal of $\OK$ defined as
		$$\fN = \prod_{\fp} \fp^{f_{\fp}},$$
		where
		\begin{equation}
			f_{\fp}=\begin{cases}
				0 & \text{if $E$ has good reduction at $\fp$},\\
				1 & \text{if $E$ has multiplicative reduction at $\fp$},\\
				2 + \delta_{\fp} & \text{if $E$ has additive reduction at $\fp$.}\\ 
			\end{cases}
		\end{equation}
		If $\fp \nmid 6$, then $\delta_{\fp} = 0$.  If $\fp \mid 6$, then $0 \leq \delta_{\fp} \leq 6\nup(2) + 3\nup(3).$  If $E$ has additive reduction at $\fp$ with $\delta_{\fp} = 0$, then $E$ is \textit{tamely ramified} at $\fp$.  If $E$ has additive reduction at $\fp$ with $\delta_{\fp} > 0$, then $E$ has \textit{wild ramification} at $\fp.$ 
	\end{definition}
	
	In particular, wild ramification can only occur at primes above 2 and 3. The conductor and minimal discriminant ideals are divisible by the same primes, but the exponents at said primes may differ.
	
		
		


		\subsection{Discriminant Twins}
		\label{sec:disctwindef}
		We can now define the object of study, discriminant (ideal) twins.
		
		\begin{definition}
			\label{def:disctwins}
			Let $K$ be a number field, and let $E$ and $E'$ be elliptic curves defined over $K$ that are not $K$-isomorphic. We say that $E$ and $E'$ are \textit{discriminant ideal twins} if they have the same discriminant ideal and the same conductor.    If, additionally, for each prime $\fp$ there exist $\fp$-minimal models for $E$ and $E'$ defined over $\OK$ such that $\Delta_{\fp} = \Delta'_{\fp}$,    then we say $E$ and $E'$ are \textit{discriminant twins.}
		\end{definition}
		
		Note that when two elliptic curves $E$ and $E'$ over $K$ are discriminant ideal twins, their Weierstrass classes are the same, $[\mathfrak{a}_{E}] = [\mathfrak{a}_{E'}]$. Thus, $E$ and $E'$ will either both have or fail to have global minimal models. Note that in \cite{Deines2018}, discriminant twins are defined only over $\Q$, which has class number one. In this early definition, they are defined as two elliptic curves with the same conductor and minimal discriminant. Our definition is a direct generalization to non-class number one number fields.
		
		Now that we have defined discriminant (ideal) twins, we can investigate the discriminant ideals of $p$-isogenous elliptic curves over number fields. 
		
		\subsection{Isogenies}
		\label{sec:isogenies}
		In this paper, we are interested specifically in isogenous discriminant (ideal) twins.
		Two elliptic curves $E_1$ and $E_2$ are \textit{$n$-isogenous} if there exists a surjective morphism $\varphi:E_1\rightarrow E_2$ such that $\ker\varphi\cong \mathbb{Z} /n \mathbb{Z}$. The isogeny $\varphi$ is defined over $K$ if $\ker\varphi$ is $\operatorname*{Gal}\!\left(\overline{K}/K\right)  $-invariant. If  $\varphi:E_1\rightarrow E_2$ is an $n$-isogeny defined over $K$, then the $K$-isomorphism class of the pair $(E,\ker \varphi)$ is a non-cuspidal $K$-rational point on the classical modular curve $X_0(n)$. Here, we recall that the non-cuspidal $K$-rational points of $X_0(n)$ parameterize isomorphism classes of pairs $(E,C)$ where $E$ is an elliptic curve defined over $K$ and $C$ is a cyclic subgroup of $E$ of order $n$ such that $C$ is $\operatorname*{Gal}\!\left(\overline{K}/K\right)  $-invariant.
		
		
		\subsection{Reduction types of  $p$-isogenous elliptic curves}
		\label{sec:reductiontypes}
		In this section, we discuss some results about reduction types of $p$-isogenous elliptic curves. In this direction, we consider the following result that relates the discriminants of $p$-isogenous elliptic curves over local fields. 
		
		
		\begin{theorem}[Dokchitser-Dokchitser \cite{DokchitserDokchitser2015}]
			\label{DokchitserSquared}
			Let $K_{\fp}$ be a local field and let $\nu = \nup$ be the normalized valuation of the completion of $K_{\fp}$. For a rational prime $p$, suppose that $E_1$ and $E_2$ are $p$-isogenous elliptic curves defined over $K_{\fp}$ with $j$-invariants $j_1$ and $j_2$, respectively.    Take $\delta_i = \nu(\Delta_{\fp}(E_i))$ to be the valuation of the minimal discriminant of $E_i$ at $\mathfrak{p}$. Table~\ref{DokchitserSquaredTable} describes the relationship between $\delta_1$ and $\delta_2$ in terms of the reduction type of $E_i$.
		\end{theorem}
		
		{\renewcommand*{\arraystretch}{1.2} \begin{longtable}{ll}
				\caption{Reduction types of $p$-isogenous elliptic curves $E_1$ and $E_2$ over $K_\fp$ and their minimal discriminant valuations $\delta_i$}\\
				\hline
				Reduction type of $E_i$   $\qquad$  $\qquad$ $\qquad$   &   $\delta_1, \delta_2$ \\
				\hline
				\endfirsthead
				\caption{The elliptic curve $\Ci:y^2=x^3+d^2 A_{p,i}(t)x+d^3 B_{p,i}(t)$}\\
				\hline
				Reduction type of $E_i$   $\qquad$  $\qquad$ &   $\delta_1, \delta_2$ \\
				\hline
				\endhead
				\hline
				\multicolumn{2}{r}{\emph{continued on next page}}
				\endfoot
				\hline
				\endlastfoot
				
				good   &$\delta_1 = \delta_2 = 0$ \\ \hline
				multiplicative  & $\delta_1 = p \delta_2$  \\
				& $\delta_2 = p \delta_2$ \\ \hline
				additive potentially multiplicative  & $\delta_2 = \delta_1 + \frac{p-1}{p} \nu(j_1)$ \\
				& $\delta_2 = \delta_1 + (p-1) \nu(j_1)$ \\ \hline
				additive potentially good $\nu(p) = 0$ &  $\delta_1 = \delta_2$ \\ \hline
				additive potentially good $\nu(p) > 0$ &  \\ 
				\hspace{4mm} potentially ordinary  &  $\delta_1 = \delta_2 $ \\
				\hspace{4mm} potentially supersingular tame &  $\delta_2 = 12 - \delta_1$ \\
				\hspace{4mm} potentially supersingular wild &  \hspace{3mm} ? 
				\label{DokchitserSquaredTable}
		\end{longtable}}
		As a consequence of Theorem~\ref{DokchitserSquared}, we obtain the following corollary pertaining to $p$-isogenous elliptic curves:
		
		\begin{corollary}
			\label{DokLemma}
			Let $E_{1}$ and $E_{2}$ be $p$-isogenous elliptic curves defined over a number
			field $K$ with potentially good reduction at $\mathfrak{p}$. If
			$\nu_{\mathfrak{p}}(p)=0$, then the $\mathfrak{p}$-adic valuation of the
			minimal discriminants of $E_{1}$ and $E_{2}$ at $\mathfrak{p}$ are equal, and
			they have the same Kodaira-N\'{e}ron type at $\mathfrak{p}$.
		\end{corollary}

		\begin{proof}
			Let $\delta_{1}$ and $\delta_{2}$ denote the $\mathfrak{p}$-adic valuations of
			the $\mathfrak{p}$-minimal discriminants of $E_{1}$ and $E_{2}$, respectively.
			By assumption, $E_{1}$ and $E_{2}$ have potentially good reduction at
			$\mathfrak{p}$ and $\nu_{\mathfrak{p}}(p)=0$. From Theorem
			\ref{DokchitserSquared}, we deduce that $\delta_{1}=\delta_{2}$. In addition,
			\cite[Table 1]{DokchitserDokchitser2015} establishes that, in this case, the
			Kodaira-N\'{e}ron types of $E_{1}$ and $E_{2}$ are the same.
		\end{proof}

		Next, we state a result that gives a criterion for elliptic curves to have potentially good reduction. 
		\begin{proposition}[{\cite[Sect.\ VII, Proposition~5.5]{Silverman1994}}]
			\label{prop:integralj}
			An elliptic curve $E$ defined over $K_{\fp}$ has potentially good reduction if and only if its $j$-invariant is integral.
		\end{proposition}
		
		Let $E$ be an elliptic curve over a number field $K$. We say that $E$ has \textit{everywhere potentially good reduction} if $\nup(j(E))\geq 0$ for each prime $\fp$ of $K$. In particular, $j(E) \in \OK$.
		It is useful to recall which Kodaira-N\'eron types occur with potentially good reduction.
		
		\begin{proposition}[{\cite[pg. 365]{Silverman2009}, \cite[pg. 42]{Lorenzini}}]
			\label{Lorenzini}
			Let $p \neq 2$ be a prime, $K$ be a number field, and $\fp$ a prime of $K$ dividing $p$. If $E$ is an elliptic curve defined over $K$ with additive reduction at $\fp$, then $E$ has potentially good reduction at $\fp$ if and only if its Kodaira-N\'{e}ron type at $\fp$ is one of ${\rm II}, {\rm II}^*, {\rm III}, {\rm III}^*, {\rm IV}, {\rm IV}^*$ or ${\rm I}_0^*$, and $E$ has potentially multiplicative reduction if and only if $\operatorname{typ}_{\fp}(E)={\rm I}_n^*$ for some $n > 0$.
		\end{proposition}
		
		When $\fp$ divides $2$, this is not the case.  There are elliptic curves with potentially good additive reduction and Kodaira-N\'{e}ron type ${\rm I}_n^*$ for $n > 0$ at primes above $2$.  For example, the curve with LMFDB label \href{https://www.lmfdb.org/EllipticCurve/Q/32/a/4}{32.a4} has 
		potentially good reduction since its $j$-invariant is $1728$, yet its Kodaira-N\'eron type is ${\rm I}_3^*$ at $p = 2$. Also, see Example~\ref{lemmafail}.
		
		Let $E_{1}$ and $E_{2}$ be $p$-isogenous discriminant ideal twins over a number field $K$. The following result shows that $E_{1}$ and $E_{2}$ have everywhere potentially good reduction.
		\begin{lemma}\label{lemmapotgoodred}
			Let $E_1$ and $E_2$ be $p$-isogenous discriminant ideal twins over some number field~$K$ for some rational prime $p.$ Then $E_1$ and $E_2$ have everywhere potentially good reduction. In particular, $j(E_1),j(E_2) \in \OK$.
		\end{lemma}
		
		\begin{proof}
			As $E_1$ and $E_2$ are $p$-isogenous elliptic curves, they have the same conductor, and thus, at every prime, they have the same reduction type. As they have the same minimal discriminant ideals, they have the same minimal discriminant valuations at all primes. Then it follows from Theorem~\ref{DokchitserSquared} that their reduction types are either good or additive potentially good reduction at each prime $\fp$. Thus, $E_1$ and $E_2$ have everywhere potentially good reduction by Proposition~\ref{prop:integralj}. Consequently, $j(E_1),j(E_2) \in \OK$.
		\end{proof}
		
		\subsection{Isogenous families of elliptic curves}
		\label{sec:isofamily}
		Let $p$ be a prime such that the classical modular curve $X_0(p)$ has genus zero. Then $p=2,3,5,7,$ or $13$. For these primes, let $\Ci$ for $i=1,2$ be as defined in Table~\ref{ta:curves}. The following result explicitly classifies all $p$-isogenous elliptic curves over a field of characteristic zero.
		
		\begin{theorem}[{Barrios \cite[Theorem~1]{Bariso}}] \label{parafamilies}
			Let $p\in\left\{  2,3,5,7,13\right\}  $ and let $K$ be a number field or local field with ring of integers $\OK$. Let $E_1$ and $E_2$ be elliptic curves defined over $K$ such that the $j$-invariants of $E_1$ and $E_2$ are not both identically $0$ or $1728$. If $E_1$ and $E_2$ are $p$-isogenous elliptic curves over $K$, then there are $t_0 \in K$ and $d_0 \in \OK$ such that $E_1$ and $E_2$ are $K$-isomorphic to $\Coo$ and $\Cto$, respectively.
		\end{theorem}
		
		\begin{remark}
			We note that \cite[Theorem 1]{Bariso} and Theorem~\ref{parafamilies} differ slightly in that the latter is stated more generally for any field $K$ of characteristic relatively prime to $6p$, and $d_0 \in K^\times/(K^\times)^2$. In this article, we will only consider number fields or local fields. In this setting, if $d_0 \in K^\times/(K^\times)^2$ and $d_0 \not \in \mathcal{O}_K$, then we may take a different representative $d_0' \in \mathcal{O}_K$ such that $\mathcal{C}_{p,i}(t_0,d_0)$ is $K$-isomorphic to $\mathcal{C}_{p,i}(t_0,d_0')$.
		\end{remark}
		
		{\renewcommand*{\arraystretch}{1.2} \begin{longtable}{ccC{2.05in}C{3.0in}}
				\caption{The elliptic curve $\Ci:y^2=x^3+d^2 A_{p,i}(t)x+d^3 B_{p,i}(t)$}\\
				\hline
				$p$ & $i$ & $A_{p,i}(t)$ & $B_{p,i}(t)$\\
				\hline
				\endfirsthead
				\caption{The elliptic curve $\Ci:y^2=x^3+d^2 A_{p,i}(t)x+d^3 B_{p,i}(t)$}\\
				\hline
				$p$ & $i$ & $A_{p,i}(t)$ & $B_{p,i}(t)$ \\
				\hline
				\endhead
				\hline
				\multicolumn{4}{r}{\emph{continued on next page}}
				\endfoot
				\hline
				\endlastfoot

				$2$ & $1$ & $  -27 (64 + t) (256 + t) $ & $   -54 (-512 + t)  (64 + t)^2  $ \\\cmidrule{2-4}
				& $2$ & $ -432 (16 + t) (64 + t)  $ & $ -3456 (-8 + t) (64 + t)^2  $ \\\hline
				
				$3$ & $1$ & $   -3 (243 + t) (27 + t)^3 $ & $    -2 (27 + t)^4 (19683 + 486t - t^2) $   \\\cmidrule{2-4}
				& $2$ & $ -243 (  3 + t) (27 + t)^3 $ & $ -1458 (27 + t)^4 (   27 - 18t - t^2) $ \\\hline
				
				$5$ & $1$ & $    -27 (125 + 22t + t^2) (3125 + 250t + t^2)  $ & $     -54 (15625 + 500t - t^2) (125 + 22t + t^2)^2 $\\\cmidrule{2-4}
				& $2$ & $ -16875 (  5 + 10t + t^2) ( 125 +   22t + t^2)  $ & $ -843750 (    1 -   4t - t^2) (125 + 22t + t^2)^2 $\\\hline
				
				$7$ & $1$ & $    -27 (49 + 13t + t^2) (2401 + 245t + t^2) $ & $      -54 (49 + 13t + t^2) (823543 + 235298t + 21609t^2 + 490t^3 - t^4) $  \\\cmidrule{2-4}
				& $2$ & $ -64827 ( 1 +  5t + t^2) (  49 +  13t + t^2) $ & $ -6353046 (49 + 13t + t^2) (     7 -     70t -    63t^2 -  14t^3 - t^4)  $  \\\hline
				
				$13$ & $1$ & $        -27 (13 + 5t + t^2) (13 + 6t + t^2)   (28561 + 15379t + 3380t^2 + 247t^3 + t^4) $ 
				& $        -54 (13 + 5t + t^2) (13 + 6t + t^2)^2 (4826809 + 3712930t + 1313806t^2 + 237276t^3 + 20618t^4 + 494t^5 - t^6)  $  \\\cmidrule{2-4}
				& $2$ & $    -771147 (13 + 5t + t^2) (13 + 6t + t^2)   (    1 +    19t +   20t^2 +   7t^3 + t^4) $  
				& $ -260647685 (13 + 5t + t^2) (13 + 6t + t^2)^2 (      1 -      38t -     122t^2 -    108t^3 -    46t^4 -  10t^5 - t^6) $
				\label{ta:curves}
		\end{longtable}}
		
		Table~\ref{ta:invariants} gives the $j$-invariants and discriminants of $\mathcal{C}_{p,i}(t,d)$, which will be helpful in the work that follows. As $d$ is a twisting parameter, we see that $j_{p, i}(t) = j_{p, i}(t, d)$ does not depend on $d.$
		
		{\renewcommand*{\arraystretch}{1.4} 
			\begin{longtable}{ccC{3.35in}C{2.1in}}
				\caption{The quantities $j_{p,i}(t)$ and $2^{-12} 3^{-12}d^{-6} \Delta_{p,i}(t,d)$} \label{ta:invariants}\\
				\hline
				$p$ & $i$ & $j_{p,i}(t, d)$ & $2^{-12} 3^{-12}d^{-6} \Delta_{p,i}(t,d)$\\
				\hline
				\endfirsthead
				\caption{The quantities $j_{p,i}(t, d)$ and $\Delta_{p,i}(t,d)d^{-6} 2^{-12} 3^{-12}$}\\
				\hline
				$p$ & $i$ & $j_{p,i}(t, d)$ & $\Delta_{p,i}(t,d)d^{-6} 2^{-12} 3^{-12}$ \\
				\hline
				\endhead
				\hline
				\multicolumn{4}{r}{\emph{continued on next page}}
				\endfoot
				\hline
				\endlastfoot
				
				$  2 $ & $ 1 $ & $ \frac{1}{t^2}(t + 256)^3
				$ & $ t^2 (t + 64)^3 $ \\ \cmidrule{2-4}
				& $ 2 $ & $ \frac{1}{t}(t + 16)^3
				$ & $ 2^{12} t   (t + 64)^3 $ \\ \hline
				$  3 $ & $ 1 $ & $ \frac{1}{t^3}(t + 243)^3 (t + 27)
				$ & $ 3^{-6} t^3 (t + 27)^8 $ \\ \cmidrule{2-4}
				& $ 2 $ & $ \frac{1}{t}(t +   3)^3 (t + 27)
				$ & $ 3^{6}  t (t + 27)^8 $ \\ \hline
				$  5 $ & $ 1 $ & $ \frac{1}{t^5}(t^2 + 250t + 3125)^3
				$ & $ t^5 (t^2 + 22t + 125)^3 $ \\ \cmidrule{2-4}
				& $ 2 $ & $ \frac{1}{t}(t^2 + 10t + 5)^3
				$ & $ 5^{12} t (t^2 + 22t + 125)^3 $ \\ \hline
				$  7 $ & $ 1 $ & $ \frac{1}{t^7}(t^2 + 245t + 2401)^3 (t^2 + 13t + 49)
				$ & $ t^7 (t^2 + 13t + 49)^2 $ \\ \cmidrule{2-4}
				& $ 2 $ & $ \frac{1}{t}(t^2 + 5t + 1)^3 (t^2 + 13t + 49)
				$ & $ 7^{12} t (t^2 + 13t + 49)^2 $ \\ \hline
				$ 13 $ & $ 1 $ & $ \frac{1}{t^{13}}(t^4 + 247t^3 + 3380t^2 + 15379t + 28561)^3 (t^2 + 5t + 13)
				$ & $ t^{13} (t^2 + 6t + 13)^3 (t^2 + 5t + 13)^2 $ \\ \cmidrule{2-4}
				& $ 2 $ & $ \frac{1}{t}(t^4 + 7t^3 + 20t^2 + 19t + 1)^3 (t^2 + 5t + 13)
				$ & $ 13^{12} t (t^2 + 6t + 13)^3 (t^2 + 5t + 13)^2 $ 	\end{longtable}}
		As it will be useful later, we note that 
		\begin{equation}\label{eq:Delatap12}
			\frac{\Delta_{p, 1}(t, d)}{\Delta_{p, 2}(t, d)} = t^{p-1} p^{-12}. 
		\end{equation}
		Further, we see that $\Co$ and $\Ct$ are both defined exactly when their discriminants are non-zero. 
		We note that their discriminants are $0$ precisely at those $t$ values which result in $j_{p,1}(t)=j_{p,2}(t) \in \{0,1728\}$ \cite[Lemma~3.3]{Bariso}. In particular, the families $\mathcal{C}_{p,i}$ fail to parameterize $p$-isogenous elliptic curves whose $j$-invariants are both identically $0$ or $1728$. Consequently, we will use a different approach for these special cases; see Section~\ref{Sec:special-jinv} below.
		
		\begin{remark} \label{SameJ}
			Finally, the set $\{j_{p, 1}(t_0), j_{p, 2}(t_0)\}$ can be same for different values of $t_0$, i.e., there are $t_0 \neq t_0'$ such that $j_{p,1}(t_0) = j_{p, 2}(t_0')$ and $j_{p,1}(t_0') = j_{p, 2}(t_0)$.
			In particular it is quick to check that $j_{p, 1}(\pm 1) = j_{p, 2}(\pm p^{\frac{12}{p-1}})$ and that $j_{p, 1}(\pm p^{\frac{12}{p-1}}) = j_{p, 2}(\pm 1)$ for $p \in  \{2, 3, 5, 7, 13\}.$
		\end{remark}
		\section{Kodaira-N\'eron types of elliptic curves with potentially good reduction}
		\label{sec:Curvespotentialgood}
		
		By Lemma~\ref{lemmapotgoodred}, $p$-isogenous discriminant ideal twins must have everywhere potentially good reduction. In this section, we establish a few results about the Kodaira-N\'{e}ron types of elliptic curves having potentially good reduction, which will be used in our main results. To this end, let~$E$ be an elliptic curve over a number field $K$. For a prime
		$\mathfrak{p}$ of $K$, let $K_{\mathfrak{p}}$ denote the completion of $K$ at
		$\mathfrak{p}$ and let $R_\fp$ denote the ring of integers of $K_{\mathfrak{p}}$. Recall that the Kodaira-N\'{e}ron type at $\mathfrak{p}$, denoted $\operatorname{typ}_{\mathfrak{p}}(E)$, is the
		reduction type of the special fiber of the minimal proper regular model of $E$
		over $R_{\mathfrak{p}}$ \cite{Neron1964}. We use Kodaira symbols to
		describe $\operatorname{typ}_{\mathfrak{p}}(E)$. 
		
		\begin{lemma}
			\label{PotGoodRedTypes}
			\label{KNforpneq3}Let $E$ be an elliptic curve over a number field $K$.
			If $E$ has potentially good reduction at a prime $\mathfrak{p}\nmid6$ of $K$, then $\operatorname{typ}_{\mathfrak{p}}(E)$ is
			uniquely determined by $\nu_{\mathfrak{p}}(\Delta(E))\ \operatorname{mod}12$ as given in \eqref{kodairasym}.
			\begin{equation}\label{kodairasym}
				\scalebox{1}{
					\renewcommand{\arraystretch}{1.18}
					\renewcommand{\arraycolsep}{.3cm}
					$\begin{array}{ccccccccc} \hline
						\operatorname{typ}_{\mathfrak{p}}(E) &  {\rm I}_0 & {\rm II} &{\rm III}& {\rm IV} & {\rm I}_0^* & {\rm IV}^* & {\rm III}^* & {\rm II}^* \\\hline
						\nu_{\mathfrak{p}}(\Delta(E))\ \operatorname{mod}12  &  0   &  2 &   3 &  4 &     6 &    8 &     9 &  10 \\ \hline
					\end{array}$}
			\end{equation}
			Further, this uniquely determines the valuation of the minimal discriminant.
		\end{lemma}
		
		\begin{proof}
			By Proposition~\ref{Lorenzini}, $E$ has potentially good reduction at a prime $\fp\nmid 6$ if and only if $\operatorname{typ}_{\mathfrak{p}}(E)$ is one of the Kodaira symbols listed in \eqref{kodairasym}.
			So let $\mathfrak{D}^{\text{min}}$ denote
			the minimal discriminant ideal of $E$. From \cite[Tableau~I]{Papadopoulos1993}%
			, we have that $\operatorname{typ}_{\mathfrak{p}}(E)$ is uniquely
			determined by $\nup(\minD)$. In particular,
			we have that:
			\[
			\scalebox{1}{
				\renewcommand{\arraystretch}{1.18}
				\renewcommand{\arraycolsep}{.3cm}
				$\begin{array}{ccccccccc} \hline
					\operatorname{typ}_{\mathfrak{p}}(E) &  {\rm I}_0 & {\rm II} &{\rm III}& {\rm IV} & {\rm I}_0^* & {\rm IV}^* & {\rm III}^* & {\rm II}^* \\\hline
					\nup(\minD)  &  0   &  2 &   3 &  4 &     6 &    8 &     9 &  10 \\ \hline
				\end{array}$}
			\]
			The result now follows since $\nu_{\mathfrak{p}}(\Delta(E))-\nup(\minD)\in12\mathbb{Z}
			$ and by \cite[Table 4.1]{Silverman1994} we have $0 \leq \nup(\mathfrak{D}^{min}) < 12$.
		\end{proof}
		
		Now suppose that $E_{1}$ and $E_{2}$ are elliptic curves defined over $K$ with
		potentially good reduction at a prime $\mathfrak{p}$. 
		Note that we are not assuming that they are isogenous or even have the same conductor.
		Suppose further that the
		$\mathfrak{p}$-adic valuation of their discriminants are equal modulo $12$ and
		let $\mathfrak{N}_{i}$ denote the conductor ideal of $E_{i}$. If
		$\mathfrak{p}\nmid6$, then Lemma \ref{PotGoodRedTypes} asserts that $E_{1}$
		and $E_{2}$ have the same minimal discriminant valuation at $\mathfrak{p}$ and
		$\operatorname*{typ}_{\mathfrak{p}}(E_{1})=\operatorname*{typ}_{\mathfrak{p}%
		}(E_{2})$. Moreover, \cite[Tableau~I]{Papadopoulos1993} also tells us that
		$\nu_{\mathfrak{p}}(\mathfrak{N}_{1})=\nu_{\mathfrak{p}}(\mathfrak{N}_{2})$. 
		
		A result similar to Lemma \ref{PotGoodRedTypes} holds for $\mathfrak{p}|3$ if
		we first require that $\nu_{\mathfrak{p}}(\mathfrak{N}_{1})=\nu_{\mathfrak{p}%
		}(\mathfrak{N}_{2})$. For reference, see the following example.
		\begin{example}
			Let $E_{1}$ and $E_{2}$ be
			the elliptic curves given by the LMFDB label \href{http://www.lmfdb.org/EllipticCurve/Q/36/a/1}{36.a1} and
			\href{http://www.lmfdb.org/EllipticCurve/Q/108/a/1}{108.a1}. Then $E_{1}$ and $E_{2}$ have potentially good reduction at $3$. Moreover $\nu_{3}(\mathfrak{D}_{1}^{\text{min}%
			})=\mathfrak{\nu}_{3}(\mathfrak{D}_{2}^{\text{min}})=9$, but,
			$\operatorname*{typ}_{3}(E_{1})={\rm III}^*$ and $\operatorname*{typ}%
			_{3}(E_{2})={\rm IV}^{\ast}$.
		\end{example}
		Since we are ultimately interested in isogenous elliptic
		curves, and isogenous elliptic curves have the same conductor, this restriction is sufficient.
		Recall that the conductor exponent $f_{\fp}$ is computed using the following formula due to Ogg \cite{Ogg}  in terms of the valuation of the minimal discriminant ideal $\minD$ and $m_{\fp}$:
		\begin{equation}
			\label{Ogg}
			f_{\fp}=\nup(\mathfrak{N})=\nup(\minD)-m_{\fp}+1.
		\end{equation}
		

		\begin{lemma}
			\label{KNforpeq3}Let $K$ be a number field and let $\mathfrak{p}$ be a prime
			such that $\mathfrak{p}|3$. Suppose further that $E_{1}$ and $E_{2}$ are
			elliptic curves over $K$ with potentially good reduction at $\mathfrak{p}$
			such that $\nu_{\mathfrak{p}}(\mathfrak{N}_{1})=\nu_{\mathfrak{p}%
			}(\mathfrak{N}_{2})$. If $\nu_{\mathfrak{p}}(\Delta(E_{1}))\equiv
			\nu_{\mathfrak{p}}(\Delta(E_{2}))\ \operatorname{mod}12$, then the
			$\mathfrak{p}$-adic valuation of the minimal discriminant of $E_{1}$ and
			$E_{2}$ at $\mathfrak{p}$ are equal and $\operatorname*{typ}_{\mathfrak{p}%
			}(E_{1})=\operatorname*{typ}_{\mathfrak{p}}(E_{2})$.
		\end{lemma}
		\begin{proof}
			From Proposition~\ref{Lorenzini}, $\operatorname{typ}_{\fp}(E_i) \in \{{\rm I_0, II, III, IV, I_0^*, IV^*, III^*, II^*}\}$. Now let $\mathfrak{D}_{i}^{\text{min}}$ denote the minimal discriminant ideal
			of $E_{i}$. In particular, $\nu_{\mathfrak{p}}(\mathfrak{D}_{i}^{\text{min}%
			})\equiv\nu_{\mathfrak{p}}(\Delta(E_{i}))\ \operatorname{mod}12$. By Ogg's
			formula (see \eqref{Ogg}), we have that $\nu_{\mathfrak{p}}(\mathfrak{N}_{\mathfrak{i}}%
			)=\nu_{\mathfrak{p}}(\mathfrak{D}_{i}^{\text{min}})-m_{\mathfrak{p},i}+1$.
			Since $\nu_{\mathfrak{p}}(\mathfrak{N}_{1})=\nu_{\mathfrak{p}}(\mathfrak{N}%
			_{2})$ and $\nu_{\mathfrak{p}}(\Delta(E_{1}))\equiv\nu_{\mathfrak{p}}%
			(\Delta(E_{2}))\ \operatorname{mod}12$, we have by Ogg's formula that%
			\begin{equation}\label{applyogg}
				\nu_{\mathfrak{p}}(\mathfrak{D}_{1}^{\text{min}})-m_{\mathfrak{p},1}%
				=\nu_{\mathfrak{p}}(\mathfrak{D}_{2}^{\text{min}})-m_{\mathfrak{p},2}%
				\qquad\Longrightarrow\qquad m_{\mathfrak{p},1}\equiv m_{\mathfrak{p}%
					,2}\ \operatorname{mod}12.
			\end{equation}
			It follows that $m_{\mathfrak{p},1}=m_{\mathfrak{p},2}$ since $1\leq
			m_{\mathfrak{p},i}\leq9$ \cite[Table 4.1]{Silverman1994}. From
			(\ref{applyogg}), we obtain that $\nu_{\mathfrak{p}}(\mathfrak{D}_{1}^{\text{min}}%
			)=\nu_{\mathfrak{p}}(\mathfrak{D}_{2}^{\text{min}})$. 
			
			From loc. cit., we also have that $m_{\mathfrak{p},i}$ uniquely determines
			$\operatorname*{typ}_{\mathfrak{p}}(E_{i})$ if $m_{\mathfrak{p},i}\neq1$. In
			particular, $\operatorname*{typ}_{\mathfrak{p}}(E_{1})=\operatorname*{typ}%
			_{\mathfrak{p}}(E_{2})$ if $m_{\mathfrak{p},i}\neq1$. So it remains to
			consider the case when $m_{\mathfrak{p},1}=m_{\mathfrak{p},2}=1$. In this case
			$\operatorname*{typ}_{\mathfrak{p}}(E_{i})\in\left\{{\rm  I_{0}^*,II}\right\}  $.
			From \cite[Tableau~II]{Papadopoulos1993}, we have that%
			\[
			\operatorname*{typ}\nolimits_{\mathfrak{p}}(E_{i})=\left\{
			\begin{array}
				[c]{cl}%
				{\rm I_{0}^*} & \text{if }\nu_{\mathfrak{p}}(\mathfrak{N}_{\mathfrak{i}})=2,\\
				{\rm II} & \text{if }\nu_{\mathfrak{p}}(\mathfrak{N}_{\mathfrak{i}})=3,4,5.
			\end{array}
			\right.
			\]
			Since $\nu_{\mathfrak{p}}(\mathfrak{N}_{1})=\nu_{\mathfrak{p}}(\mathfrak{N}%
			_{2})$, we conclude that $\operatorname*{typ}_{\mathfrak{p}}(E_{1}%
			)=\operatorname*{typ}_{\mathfrak{p}}(E_{2})$.
		\end{proof}
		
		\begin{remark}
			Lemma \ref{KNforpeq3} does not hold for primes $\mathfrak{p}$ dividing $2$.
			Suppose further that $\nu_{\mathfrak{p}}(\mathfrak{N}_{1})=\nu_{\mathfrak{p}%
			}(\mathfrak{N}_{2})$ and $\nu_{\mathfrak{p}}(\mathfrak{D}_{1}^{\text{min}%
			})=\nu_{\mathfrak{p}}(\mathfrak{D}_{2}^{\text{min}})$. From \cite[Tableau
			V]{Papadopoulos1993} we see that the $\mathfrak{p}$-adic valuation of the
			minimal discriminant and conductor exponent does not uniquely determine the
			Kodaira-N\'{e}ron type at $\mathfrak{p}$ since the Kodaira-N\'{e}ron type $\rm{I}_{n}%
			^{\ast}$ for $n\geq1$ can now occur. This phenomenon is depicted in the
			following example:
		\end{remark}
		
		\begin{example}
			\label{lemmafail}Let $K=\mathbb{Q}(i)$ and let $\mathfrak{p=(}1+i)\mathcal{O}_{K}$ be the prime above $2$, in particular $\fp^2 = 2\mathcal{O}_K$. Let $E_{1}$ and $E_{2}$ be
			the elliptic curves given by the LMFDB label \href{https://www.lmfdb.org/EllipticCurve/2.0.4.1/512.1/a/1}{2.0.4.1-512.1-a1} and
			\href{https://www.lmfdb.org/EllipticCurve/2.0.4.1/512.1/a/2}{2.0.4.1-512.1-a2}. Then $E_{1}$ and $E_{2}$ are isogenous everywhere potentially good elliptic curves. Moreover $\nu_{\mathfrak{p}} (\mathfrak{D}_{1}^{\text{min}})=\mathfrak{\nu}_{\mathfrak{p}}(\mathfrak{D}_{2}^{\text{min}})=17$, but $\operatorname*{typ}_{\mathfrak{p}}(E_{1})=\rm{II}^{\ast}$ and $\operatorname*{typ}_{\mathfrak{p}}(E_{2})=\rm{I}_{4}^{\ast}$. This shows that the assumptions that the $\mathfrak{p}$-adic valuations of the conductor and minimal discriminant ideal are not enough to guarantee the same Kodaira-N\'{e}ron type at $\mathfrak{p}$.
		\end{example}

		\section{Discriminant ideal twins for identical \texorpdfstring{$j$}{j}-invariants}
		\label{Sec:special-jinv}
		
		In this section, we categorize all $p$-isogenous discriminant (ideal) twins where the curves also share the same $j$-invariant for $p \in \{2, 3, 5, 7, 13\}$. To this end, we must consider two cases as the parameterized families of elliptic curves $\mathcal{C}_{p, i}(t,d)$ fail to capture the case when two $p$-isogenous elliptic curves share the same $j$-invariant $0$ or $1728$. For this reason, we treat the case corresponding to $j$-invariants $0$ or $1728$ separately. 
		
		The following lemma tells us that $p$-isogenous discriminant ideal twins sharing the same $j$-invariant will necessarily have complex multiplication (CM).  Further, if we base change to a field containing their CM endomorphisms their isogeny class collapses, the curves are now isomorphic and no longer discriminant ideal twins.
		
		\begin{lemma}
			\label{SameJCM}
			Let $K$ be a field and suppose that $E_1$ and $E_2$ are non-$K$-isomorphic isogenous elliptic curves. If $E_1$ and $E_2$ have the same $j$-invariant, 
			then they have complex multiplication. 
			Further, such curves are isomorphic exactly over the fields containing their complex multiplication field.
		\end{lemma}
		
		\begin{proof}
			To create a contradiction, assume $E_1$ and $E_2$ are non-$K$-isomorphic isogenous elliptic curves over $K$ without CM but with the same $j$-invariant.
			Let $\varphi:E_1 \rightarrow E_2$ be the isogeny defined over $K$.
			As $E_1$ and $E_2$ have the same $j$-invariant, then $E_1$ and $E_2$ are isomorphic over an algebraic closure $\overline{K}$.  Denote by $\tau:E_2 \rightarrow E_1$ the isomorphism defined over $\overline{K}$.
			Then $\tau\circ \varphi: E_1 \rightarrow E_1$ is an endomorphism of $E_1.$
			As the curves do not have CM, then $\tau \circ \varphi = [n ]$ is multiplication-by-$n$ for some $n \in \mathbb{Z}.$
			Then for any $\sigma \in \text{Gal}(\overline{K}/K)$, $\sigma$ commutes with $\varphi$ and $[n]$ and thus $\sigma$ also commutes with $\tau$.  Consequently, $\tau$ was actually defined over $K$, and hence $E_1$ and $E_2$ are isomorphic over $K$.
			
			Now assume $E_1$ and $E_2$ are non-$K$-isomorphic isogenous curves over $K$ with CM by the order $\mathcal{O}$, and that they have the same $j$-invariant.  Using the notation as above, $\tau \circ \varphi$ is an endomorphism of $E_1$ and thus can be viewed as an element of $\mathcal{O},$ but it is only defined as an endomorphism over an extension of $K$ containing $\mathcal{O}$.  Thus by the same argument as above $E_1$ is isomorphic to $E_2$ over some field $L$ if and only if $\mathcal{O} \subset L.$
		\end{proof}
		
		As we are studying isogenous elliptic curves with the same $j$-invariants, and thus with complex multiplication, we will only find discriminant (ideal) twins in number fields that do not contain the CM order.
		
		\subsection{Singular $j$-invariants}
		We study curves with both $j$-invariants $0$ or $1728$ first.
		
		Here, we will find that for $j = 0$, we can only have $3$-isogenous discriminant ideal twins, and for $j = 1728$, we can only have $2$-isogenous discriminant ideal twins, both over number fields not containing the complex multiplication field of the elliptic curves (see Example~\ref{Exampledisctwinsspecialj}). We also find that discriminant twins only occur for $3$-isogenies under the further assumption that the number field contains $\sqrt{3}$. While $2$-isogenous discriminant ideal twins occur, $2$-isogenous discriminant twins do not exist. Thus, we start by narrowing down the isogeny degrees we need to examine.
		
		Before diving in, we recall some facts about modular polynomials and $j$-invariants; see \cite[Chapter II]{Silverman1994}. Let $p$ be a rational prime and $\Phi_{p}(x, y)$ be the classical modular polynomial for the modular curve $X_0(p)$.  Solutions $(X, Y)$ to $\Phi_p(x, y) = 0$ correspond to $p$-isogenous elliptic curves $E_1$ and $E_2$ such that $j(E_1)=X$ and $j(E_2)=Y$. In particular, let $E$ be an elliptic curve defined over a number field $K$, then $\Phi_p(j(E), y)$ is a degree $p+1$ polynomial in $y$, the roots of which are exactly the $j$-invariants of the $p+1$ elliptic curves over $\overline{K}$ that are $p$-isogenous  to $E$. Thus, to find all $p$ such that there exist $p$-isogenies where both $j$-invariants are $0$ or $1728$, we factor the modular polynomials $\Phi_p(0, y)$ and $\Phi_p(1728, y)$, respectively.
		
		\begin{lemma}\label{numberpiso}
			Let $E$ be an elliptic curve defined over a number field $K$ with $j=j(E)\in\{0,1728\}$. If $p\in\{2,3,5,7,13\}$, then the number of $p$-isogenies $E\rightarrow E$ defined over~$\overline{K}$ are as given in~\eqref{tab:ModPoly}.
			\begin{equation}
				\renewcommand{\arraystretch}{1.3}
				\renewcommand{\arraycolsep}{0.2cm}
				\begin{array}
					[c]{cccccc} \hline
					p & 2 & 3 & 5 & 7 & 13\\\hline
					j=0 & 0 & 1 & 0 & 2 & 2\\
					j=1728 & 1 & 0 & 2 & 0 & 2\\\hline
				\end{array}
				\label{tab:ModPoly}%
			\end{equation}
		\end{lemma}
		
		\begin{proof}
			In \cite{Bruinier_ClassPolynomials,SutherlandVolcanoes}, the classical modular polynomials $\Phi_{p}(x,y)$ for $p\in\{2,3,5,7,13\}$ are given explicitly. For each $p$ and $j\in\{  0,1728\}  $, we factor the modular polynomial $\Phi_{p}(j,y)$ to find when it has a factor of $y$ or $y-1728$ if $j=0$ or $j=1728$, respectively. 
			In this direction, let $\nu_{f}$ denote the $f$-adic valuation of $\mathbb{Q}(y)$ for some irreducible element $f\in\mathbb{Q}[y]$. In \cite{GitHubDIT}, we show that $\nu_{y-j}(\Phi_{p}(j,y))$ is the quantity given in~\eqref{tab:ModPoly}. We now claim that $\nu_{y-j}(\Phi_{p}(j,y))$ gives the number of $p$-isogenies $\psi:E\rightarrow E$ defined over $\overline{K}$, such that $j=j(E)$. This is automatic if $\nu_{y-j}(\Phi_{p}(j,y))=0$, and so it remains to consider those cases in~\eqref{tab:ModPoly} where $\nu_{y-j}(\Phi_{p}(j,y))>0$. In these cases, Tsukazaki shows that there are $\nu_{y-j}(\Phi_{p}(j,y))$ $p$-isogenies \cite[Section 4.3]{Tsukazaki}.
		\end{proof}
		
		Factoring the modular polynomials as in the proof of Lemma~\ref{numberpiso} can lead to over-counting the number of non-trivial isogenies. The $j$-invariants $0$ and $1728$ are CM $j$-invariants. 
		Elliptic curves with $j$-invariant $0$ have CM by $\mathbb{Z}[  \zeta_{3}]  $, where $\zeta_{3}=\frac{-1+\sqrt{-3}}{2}$. Similarly, elliptic curves with $j$-invariant $1728$ have CM by $\mathbb{Z}[  i]  $, where $i=\sqrt{-1}$. When the field of definition of an elliptic curve with CM contains the CM field, some isogenies are actually endomorphisms, giving an isomorphic elliptic curve instead of another distinct curve. The next example illustrates this phenomenon.
		
		\begin{example}
			Consider the elliptic curves $E_{1}:y^{2}=x^{3}-108$ and $E_{2}:y^{2}=x^{3}+4$, whose LMFDB labels are \href{https://www.lmfdb.org/EllipticCurve/Q/108.a1/}{108.a1} and \href{https://www.lmfdb.org/EllipticCurve/Q/108.a2/}{108.a2}, respectively. Over $\mathbb{Q}$, these elliptic curves are $3$-isogenous, and the isogeny between them is not a $\mathbb{Q}$-isomorphism. Both elliptic curves have $j$-invariant~$0$ and thus have CM by $\mathbb{Z}[  \zeta_{3}]  $. Now let $K=\mathbb{Q}(\zeta_{3})$ be the CM field, and consider the base change of $E_{1}$ and $E_{2}$ to $K$. Now, the two elliptic curves are $K$-isomorphic as there is an isomorphism $\tau = [\zeta_3+2,0,0,0] $ from $E_1$ to $E_2$. The LMFDB label of this elliptic curve over $K$ is \href{https://www.lmfdb.org/EllipticCurve/2.0.3.1/1296.1/CMa/1}{1296.1-CMa1}. We note that the $K$-isogeny class of $E_{1}$ is a singleton set consisting of the $K$-isomorphism class of $E_{1}$.
		\end{example}
		\begin{proposition}\label{Proponjs}
			Let $E_{j}$ be an elliptic curve defined over a number field $K$ with $j$-invariant $j\in\left\{  0,1728\right\}  $. For $p\in\left\{5,7,13\right\}  $, if $E_{j}$ admits a $p$-isogeny defined over $K$ to an elliptic curve with the same $j$-invariant, then the elliptic curves are $K$-isomorphic.
		\end{proposition}
		
		\begin{proof}
			By Lemma \ref{numberpiso}, the number of $p$-isogenies defined over $\overline{K}$ admitted by $E_j$ are as given in~\eqref{tab:ModPoly}. Tsukazaki computed the kernel polynomial of these $p$-isogenies, and they are explicitly given in \cite[Tables 4.2 and 4.3]{Tsukazaki}. In particular, the field of definition of the kernel polynomials are as given in \eqref{tab:IsogField}.
			\begin{equation}
				\renewcommand{\arraystretch}{1.3}
				\renewcommand{\arraycolsep}{0.2cm}
				\begin{array}
					[c]{cccc}\hline
					p  & 5 & 7 & 13\\\hline
					j=0 & - &\mathbb{Q}(\zeta_{3}) &\mathbb{Q}(\zeta_{3})\\\hline
					j=1728 &\mathbb{Q}(i) & - &\mathbb{Q}(i)\\\hline
				\end{array}
				\label{tab:IsogField}%
			\end{equation}
			Consequently, if $E_j$ admits a $p$-isogeny defined over $K$, then $\zeta_3 \in K$ (resp. $i \in K$) for $j=0$ (resp. $1728$).
			
			Next, let $\mathcal{O}_{j}$ denote the endomorphism ring of $E_{j}$, i.e., $\mathcal{O}_{0}=\mathbb{Z}\lbrack\zeta_{3}]$ and $\mathcal{O}_{1728}=\mathbb{Z}\lbrack i]$. Observe that $p\mathcal{O}_{j}$ has the following prime ideal factorizations in $\mathcal{O}_{j}$:
			\begin{equation}
				\renewcommand{\arraystretch}{1.3}
				\renewcommand{\arraycolsep}{0.2cm}
				\begin{array}
					[c]{cccc}\hline
					p  & 5 & 7 & 13\\\hline
					p\mathbb{Z}\lbrack\zeta_{3}] &
					\left(  5\right)   & \left(  \zeta_{3}+3\right)  \left(  \zeta_{3}-2\right)
					& \left(  \zeta_{3}+4\right)  \left(  \zeta_{3}-3\right)  \\\hline
					p\mathbb{Z}\lbrack i] &  \left(
					2+i\right)  \left(  2-i\right)   & \left(  7\right)   & \left(  3+2i\right)
					\left(  3-2i\right) \\\hline
				\end{array}
				\label{tab:CMendomorphisms}
			\end{equation}
			Note that $E_j$ admits a $p$-isogeny exactly when $p\mathcal{O}_{j}=\mathfrak{p}_{1}\mathfrak{p}_{2}$ splits in $\mathcal{O}_{j}$. Further, for each prime $\fp_k$ that splits, since $\left\vert \mathcal{O}_{j}/\mathfrak{p}_{k}\right\vert =p$, $E_{j}$ admits a $p$-endomorphism.
			Letting $[a]$ represent the multiplication-by-$a$ map, these endomorphisms are given by $[a] + [b]\circ[\zeta_3]$ and $[a] + [b]\circ[i]$, respectively, where $\mathfrak{p}_k = (a + b\zeta_3)$ or $(a + b i)$. Thus these endomorphisms are defined over $\mathcal{O}_j$, as are the isogenies. This means the two isogenies are, in fact, just endomorphisms, and the curves must be $K$-isomorphic.
		\end{proof}
		
		As a consequence of Proposition~\ref{Proponjs}, we see that it is impossible to have elliptic curves $E_1$ and $E_2$ with the same $j$-invariant $0$ or $1728$ that are $p$-isogenous discriminant ideal twins with $p\in\{5,7,13\}$. We now shift our focus to the case of $2$- and $3$-isogenies. Before proving our main result, we require the following lemma.
		
		\begin{lemma}
			\label{j0or1728} Let $E_{1}$ and $E_{2}$ be elliptic curves defined over a number field $K$. 
			Suppose that $j=j(E_{1})=j(E_{2})\in\left\{  0,1728\right\}  $ and let
			\begin{equation}\label{defofmtwist}
				m=\left\{
				\begin{array}
					[c]{cl}%
					6 & \text{if }j\! =0,\\
					4 & \text{if }j\! =1728.
				\end{array}
				\right.
			\end{equation}
			Then for $p \in \{2,3,5,7,13 \}$, $E_1$ and $E_2$ are (non-$K$-isomorphic) $p$-isogenous if and only if the following hold:
			\begin{itemize}
				\item[$(i)$] $(p,j)=(2,1728)$ and $i \not \in K$;
				\item[$(ii)$] $(p,j)=(3,0)$ and $\zeta_3 \not \in K$.
			\end{itemize}
			Further, there exists $d\in K^{\times}/\left(  K^{\times}\right)  ^{m}$  such that  $E_{i}$ is $K$-isomorphic to $E_{i,j}(d)$ where $E_{i,j}(d)$ is as given below:
			\begin{align*}
				E_{1,1728}\!\left(  d\right)   &  :y^{2}=x^{3}+dx,\qquad & E_{1,0}\!\left(
				d\right)   &  :y^{2}=x^{3}+d,\\
				E_{2,1728}\!\left(  d\right)   &  :y^{2}=x^{3}-4dx,\qquad & E_{2,0}\!\left(
				d\right)   &  :y^{2}=x^{3}-27d.
			\end{align*}
		\end{lemma}
		
		\begin{proof}
			The case when $p \in \{5,7,13\}$ was settled in Proposition~\ref{Proponjs}. So we may suppose that $p\in \{2,3\}$.
			If $j=j(E_{1})\in\left\{  0,1728\right\}  $, then $E_1$ is a twist of the elliptic curve $E_{1,j}(1)$ by \cite[Proposition~X.5.4]{Silverman2009}.
			Consequently, there exists a $d\in K^{\times}/\left(  K^{\times}\right)  ^{m}$ such that $E_{1}$ is $K$-isomorphic to $E_{1,j}\!\left(  d\right)  $. By Lemma~\ref{numberpiso}, there is exactly one $p$-isogeny admitted by $E_{1}$ to an elliptic curve of the same $j$-invariant. We write this isogeny explicitly as follows. Let
			\[
			L_{0}=\frac{x^{3}+4d}{x^{2}}\qquad\text{and}\qquad L_{1728}=\frac{x^{2}+d}{x}.
			\]
			Now observe that
			\begin{align*}
				\left(  y\frac{d}{dx}L_{0}\right)  ^{2}  & =\left(  x^{3}+d\right)  \left(
				\frac{d}{dx}L_{0}\right)  ^{2}=\left(  x^{3}+d\right)  \left(  \frac{x^{3}%
					-8d}{x^{3}}\right)  ^{2}=L_{0}^{3}-27d,\\
				\left(  y\frac{d}{dx}L_{1728}\right)  ^{2}  & =\left(  x^{3}+dx\right)
				\left(  \frac{d}{dx}L_{1728}\right)  ^{2}=\left(  x^{3}+dx\right)  \left(
				\frac{x^{2}-d}{x^{2}}\right)  ^{2}=L_{1728}^{3}-4dL_{1728}.
			\end{align*}
			It follows that the map $\psi_{j}:E_{1,j}\!\left(  d\right)  \rightarrow E_{2,j}\!\left(  d\right)  $ defined by $\psi_{j}(x,y)=\left(  L_{j},y\frac{d}{dx}L_{j}\right)  $ is an isogeny of degree $p$ where $(p,j) \in \{(2,1728),(3,0) \}$.
			
			By the above, if $E_1$ and $E_2$ are $p$-isogenous, we may assume without loss of generality that $E_i=E_{i,j}(d)$ for some $d\in K^\times$. Note that if $E_1$ and $E_2$ are $K$-isomorphic, then we have an equivalence with the existence of a non-zero $u\in K$ such that $\tau=[u,0,0,0]$ is a $K$-isomorphism from $E_{2,j}(d)\rightarrow E_{1,j}(d)$. Let $F_j$ denote the codomain of $\tau$. Then
			\[
			F_{0}:y^{2}=x^{3}-27u^{-6}d\qquad\text{and}\qquad F_{1728}:y^{2}=x^{3}%
			-4u^{-4}d.
			\]
			It follows that $E_{1,j}(d)=F_{j}$ if and only if $(i)$ $u^{6}=-27$ for $j=0$ or $(ii)$ $u^{4}=-4$ for $j=1728$. These equations hold if and only if $(i)$ $\zeta_{3}\in K$ for $j=0$ and $(ii)$ $i\in K$ for $j=1728$.
		\end{proof}

		Next, in Theorem~\ref{thmforj-_1728}, we provide necessary conditions for two elliptic curves with $j$-invariants both $0$ or $1728$ to be discriminant ideal twins.
		In the theme of this paper, we classify all $3$-isogenous discriminant ideal twins and discriminant twins, but in the $2$-isogenous case, there is no converse, and examples to this effect are given after the theorem.

		\begin{theorem}\label{thmforj-_1728}
			Let $K$ be a number field and let $E_1$ and $E_2$ be $p$-isogenous elliptic curves defined over $K$ for $p \in \{2,3,5,7,13 \}$. 
			Let $\zeta_3 = \frac{-1 + \sqrt{-3}}{2}$ be a third root of unity and $i = \sqrt{-1}$ a fourth root of unity.
			Suppose further that $j=j(E_{1} )=j(E_{2})\in\left\{  0,1728\right\}  $.
			If $E_{1}$ and~$E_{2}$ are discriminant ideal twins, then one of the following holds:
			
			\begin{enumerate}
				\item[($a$)] $(p,j)=(3,0)$, $\zeta_{3}\not \in K$, and the ideal $3\mathcal{O}_{K}$ is a square. 
				
				\item[($b$)] $(p,j)=(2,1728)$, $i\not \in K$, and the ideal $2\mathcal{O}_{K}$ is a square.
			\end{enumerate}
			Conversely,
			\begin{enumerate}
				\item[($c$)] if $(p,j)=(3,0)$, $\zeta_{3}\not \in K$, and the ideal $3\mathcal{O}_{K}$ is a square, then $E_1$ and $E_2$ are discriminant ideal twins. Further, if $3$ is a square in $K$, then $E_1$ and $E_2$ are discriminant twins. 
				\item[($d$)] if $(p,j)=(2,1728)$, then there are no discriminant twins. 
			\end{enumerate}
			%
			%
		\end{theorem}
		
		\begin{proof}
			Suppose that $E_{1}$ and $E_{2}$ are $p$-isogenous discriminant ideal twins defined over $K$ with $j=j(E_{1})=j(E_{2})\in\left\{  0,1728\right\}  $. In particular, $E_1$ and $E_2$ are not $K$-isomorphic. By Lemma~\ref{j0or1728}, we have that $\left(  p,j\right)  \in\left\{\left(  2,1728\right)  ,\left(  3,0\right)  \right\}  $ with the further condition that $i \not \in K$ if $j=1728$ and $\zeta_3 \not \in K$ if $j=0$. Moreover, $E_{i}$ is $K$-isomorphic to $E_{i,j}=E_{i,j}(d)$ for some $d\in K^{\times}/(K^{\times})^{m}$.
			In the proof below, we reference the discriminants of these elliptic curves, which are as follows:
			\[
			\begin{array}
				[c]{ccccccc}%
				\Delta(E_{1,1728}(d)) & = & -2^{6}d^{3}, & \qquad & \Delta(E_{1,0}(d)) & = &
				-2^{4}3^{3}d^{2},\\
				\Delta(E_{2,1728}(d)) & = & 2^{12}d^{3}, &  & \Delta(E_{2,0}(d)) & = &
				-2^{4}3^{9}d^{2}.%
			\end{array}
			\]
			
			For each prime $\mathfrak{p}$ of $K$, let $E_{i,j,\mathfrak{p}}$ be a $\mathfrak{p}$-minimal model of $E_{i,j}$ over $K$, given by an integral Weierstrass model. 
			Consider the inclusion $\iota: \OK \hookrightarrow R_\fp$.
			By Lemma~\ref{newSilverman}, there exists $u_{i,j,\mathfrak{p}}\in \OK \subset R_{\fp}$ such that $u_{i,j,\mathfrak{p}}^{-12}\Delta(E_{i,j})=\Delta(E_{i,j,\mathfrak{p}})$. Since $E_{1,j}$ and $E_{2,j}$ are discriminant ideal twins, there exists a $\mu_{j,\mathfrak{p}}\in \OK \cap R_{\fp}^{\times}$, i.e., $\nup(\mu_{j,\fp})=0$, such that
			\begin{align*}
				\Delta(E_{1,j,\mathfrak{p}})=\mu_{j,\mathfrak{p}}\Delta(E_{2,j,\mathfrak{p}%
				})\qquad & \Longrightarrow\qquad u_{1,j,\mathfrak{p}}^{-12}\Delta(E_{1,j})%
				=\mu_{j,\mathfrak{p}}u_{2,j,\mathfrak{p}}^{-12}\Delta(E_{2,j}).\\
				& \Longrightarrow\qquad\left(  \frac{u_{2,j,\mathfrak{p}}}{u_{1,j,\mathfrak{p}%
				}}\right)  ^{12}=\mu_{j,\mathfrak{p}}\frac{\Delta(E_{2,j})}{\Delta(E_{1,j}%
					)}=\left\{
				\begin{array}
					[c]{ll}%
					3^{6}\mu_{0,\mathfrak{p}} & \text{if }j=0,\\
					-2^{6}\mu_{1728,\mathfrak{p}} & \text{if }j=1728.
				\end{array}
				\right.
			\end{align*}
			By the above, we have that $\left(  \frac{u_{2,j,\mathfrak{p}}}{u_{1,j,\mathfrak{p}
			}}\right)  ^{2}R_{\fp}=pR_{\fp}$. Consequently, the ideal $p R_\fp$ is a square for each prime $\fp$. This, in turn, implies that the ideal $p\mathcal{O}_K$ is a square.
			
			We now show the converse of $\left(  a\right)  $. To this end, suppose $\zeta_{3}\not \in K$ and that the ideal $3\mathcal{O}_{K}$ is a square. Let $d\in K^{\times}$. We claim that $E_{1,0}(d)$ and $E_{2,0}(d)$ are $3$-isogenous discriminant ideal twins. Now observe that for each prime $\mathfrak{p}\nmid3$, the $\mathfrak{p}$-adic valuations of the minimal discriminants of $E_{1}$ and $E_{2}$ at $\mathfrak{p}$ are equal by Corollary \ref{DokLemma}. So suppose that $\mathfrak{p}|3$ and let
			$\nu_{\mathfrak{p}}(3)=e$. Then
			\[
			\nu_{\mathfrak{p}}(\Delta(E_{1,0}(d)))=3e+2d\qquad\text{and}\qquad
			\nu_{\mathfrak{p}}(\Delta(E_{2,0}(d)))=9e+2d.
			\]
			Since $3\mathcal{O}_{K}$ is a square ideal, we have that $e$ is even. Consequently, 
			$$\nu_{\mathfrak{p}}(\Delta(E_{1,0}(d)))\equiv\nu_{\mathfrak{p}}(\Delta(E_{2,0}(d)))\ \operatorname{mod}12.$$
			It follows from Lemma~\ref{KNforpeq3} that the $\mathfrak{p}$-adic valuations of the minimal discriminants of $E_{1,0}(d)$ and $E_{2,0}(d)$ are equal. This shows that $E_{1,0}(d)$ and $E_{2,0}(d)$ are discriminant ideal twins for each $d\in K^{\times}$.

			It remains to consider the case when $E_{1}$ and $E_{2}$ are discriminant twins. When this is the case, for each prime $\mathfrak{p}$ of $K$, there exists $u_{1,j,\mathfrak{p}},u_{2,j,\mathfrak{p}}\in \OK$ such that
			\begin{align*}
				u_{1,j,\mathfrak{p}}^{-12}\Delta(E_{1,j})=u_{2,j,\mathfrak{p}}^{-12}%
				\Delta(E_{2,j})\qquad & \Longrightarrow\qquad\left(  \frac{u_{2,j,\mathfrak{p}%
				}}{u_{1,j,\mathfrak{p}}}\right)  ^{12}=\frac{\Delta(E_{2,j})}{\Delta(E_{1,j}%
					)}\\
				& \Longrightarrow\qquad\left(  \frac{u_{2,j,\mathfrak{p}}}{u_{1,j,\mathfrak{p}%
				}}\right)  ^{2}=\left\{
				\begin{array}
					[c]{ll}%
					3 & \text{if }j=0,\\
					-2 & \text{if }j=1728.
				\end{array}
				\right.
			\end{align*}
			In particular, for each $\fp$, $3$ (resp. $-2$) is a square in $R_\fp$ for $j=0$ (resp. $j=1728$). We conclude by the Grunwald-Wang Theorem~\cite{MR33801} that $3$ (resp. $-2$) is a square in $K$ if $j=0$ (resp. $j=1728$).
			
			We now break into two cases to finish examining discriminant twins. 
			We first show that if $(p,j)=(3,0)$, $3$ is a square of $K$, and $\zeta_{3}\not \in K$, then $E_{1,0}(d)$ and $E_{2,0}(d)$ are $3$-isogenous discriminant twins for each $d \in K^\times$. Now let $3=z^2$ for some $z \in K$. By the proof of $(a)$, we have that $E_{1,0}(d)$ and $E_{2,0}(d)$ are discriminant ideal twins. With notation as above,
			
			\[
			\mu_{0,\mathfrak{p}}^{-1}=\frac{\Delta(E_{2,0,\mathfrak{p}})}{\Delta(E_{1,0,\mathfrak{p}%
				})}=\frac{u_{2,0,\mathfrak{p}}^{-12}\Delta(E_{2,0}(d))}{u_{1,0,\mathfrak{p}}%
				^{-12}\Delta(E_{1,0}(d))}=\left(  \frac{u_{1,0,\mathfrak{p}}}{u_{2,0,\mathfrak{p}%
			}}z\right)  ^{12}%
			\]
			since $\frac{\Delta(E_{2,0}(d))}{\Delta(E_{1,0}(d))}=3^{6}$. In particular,
			there exists $\kappa_{\mathfrak{p}}\in \OK$ with $\nup(\kappa_\fp)=0$ such that
			$\mu_{0, \mathfrak{p}}=\kappa_{\mathfrak{p}}^{12}$. Now let $E_{1,0,\mathfrak{p}%
			}^{\prime}$ be the $K$-isomorphic elliptic curve obtained from
			$E_{1,0,  \mathfrak{p}}$ via the isomorphism $\left[  \kappa_{\mathfrak{p}},0,0,0\right]  $. Since $\nup(\kappa_{\mathfrak{p}})=0$, we have
			that $E_{1,0,\mathfrak{p}}^{\prime}$ is a $\mathfrak{p}$-minimal model.
			Furthermore, we obtain the following equalities:%
			\[
			\Delta(E_{1,0,\mathfrak{p}}^{\prime})=\kappa_{\mathfrak{p}}^{-12}\Delta
			(E_{1,0,\mathfrak{p}})=\mu_{0,\mathfrak{p}}^{-1}\Delta(E_{1,0,\mathfrak{p}}%
			)=\Delta(E_{2,0,\mathfrak{p}}).
			\]
			This shows that $E_{1,0}(d)$ and $E_{2,0}(d)$ are $3$-isogenous discriminant twins for each
			$d\in K^{\times}$.
			
			It remains to show that if $(p,j)=(2,1728)$, then there are no $2$-isogenous
			discriminant twins. To this end, it suffices to show that $E_{1,1728}(d)$ and
			$E_{2,1728}(d)$ are not discriminant twins for each $d\in K^\times$. Towards a
			contradiction, suppose that $E_{1,1728}(d)$ and $E_{2,1728}(d)$ are
			$2$-isogenous discriminant twins for some $d\in K^\times$. By what was established earlier, we know that $-2$ is a
			square in $K$. If $i\in K$, then $E_{1,1728}(d)$ and $E_{2,1728}(d)$ are not
			$2$-isogenous discriminant twins by Proposition \ref{Proponjs}. So we may
			assume that $i\not \in K$. Now let $\mathfrak{p}$ be a prime above $2$ in $K$.
			Since $-2$ is a square in $K$,
			there exists $z\in\OK$ such that $-2=z^{2}$. Now let
			$E_{1,1728,\fp}^{\prime}(d)$ be the $K$-isomorphic elliptic curve
			obtained from $E_{1,1728}(d)$ via the isomorphism $\tau=[z,0,0,0]$. Then%
			\[
			\frac{\Delta(E_{1,1728,\fp}^{\prime}(d))}{\Delta(E_{2,1728}(d))}=-1.
			\]
			By assumption, $E_{1,1728}(d)$ and $E_{2,1728}(d)$ are discriminant
			twins. So there are $u_{1, 1728, \mathfrak{p}},u_{2, 1728, \mathfrak{p}}\in K^\times$ such that%
			\[
			u_{1,1728, \mathfrak{p}}^{-12}\Delta(E_{1,1728,\fp}^{\prime}(d))=u_{2, 1728, \mathfrak{p}%
			}^{-12}\Delta(E_{2,1728}(d))\quad\Longrightarrow\quad\left(  \frac
			{u_{1,1728, \mathfrak{p}}}{u_{2, 1728, \mathfrak{p}}}\right)  ^{12}=\frac{\Delta
				(E_{1,1728,\fp}^{\prime}(d))}{\Delta(E_{2,1728}(d))}=-1.
			\]
			But this is impossible since $i\not \in K$. This shows that there
			are no $2$-isogenous discriminant twins, which concludes the proof.
		\end{proof}
		
		The following example illustrates the necessity of the conditions stated in Theorem~\ref{thmforj-_1728} for discriminant twins. 
		
		\begin{example}\label{Exampledisctwinsspecialj}
			Consider the class number one number field $K=\mathbb{Q}(\sqrt[4]{6})$. Then the pairs of elliptic curves $\left(  E_{1,0}%
			(1),E_{2,0}(1)\right)  $ and $\left(  E_{1,1728}(1),E_{2,1728}(1)\right)  $
			are discriminant ideal twins. Indeed, the minimal discriminant ideals of these elliptic curves are as given below:%
			\begin{align*}
				\mathcal{D}^{\text{min}}(E_{1,0}(1))  & =\mathcal{D}^{\text{min}}%
				(E_{2,0}(1))=\left(  2\sqrt[4]{6^{3}}-3\sqrt[4]{6^{2}}+5\sqrt[4]{6}-8\right)
				^{4}\mathcal{O}_{K},\\
				\mathcal{D}^{\text{min}}(E_{1,1728}(1))  & =\mathcal{D}^{\text{min}%
				}(E_{2,1728}(1))=\left(  2\sqrt[4]{6^{3}}-3\sqrt[4]{6^{2}}+5\sqrt[4]%
				{6}-8\right)  ^{12}\mathcal{O}_{K}.
			\end{align*}
			Moreover, the elliptic curves $E_{1,1728}(1)$ and $E_{2,1728}(1)$ are not discriminant twins by Theorem~\ref{thmforj-_1728}. We also have from the theorem that $E_{1,0}(1)$ and $E_{2,0}(1)$ are not discriminant twins since
			$\sqrt{3}\not \in K$.
		\end{example}
		
		The converse to Theorem \ref{thmforj-_1728} $(b)$ is false, as illustrated by the
		following example:
		
		\begin{example}\label{failureat2specialj}
			Consider the class number one number field $K=\mathbb{Q}(\sqrt[6]{2})$.  This field contains $\sqrt{2}$, but not $i$ and thus satisfies Theorem \ref{thmforj-_1728}.  However, the elliptic curves $E_{1,1728}(1)$ and
			$E_{2,1728}(1)$ are not discriminant ideal twins since their minimal
			discriminant ideals are:%
			\[
			\mathcal{D}^{\text{min}}(E_{1,1728}(1))=\left(  \sqrt[6]{2}\right)
			^{12}\mathcal{O}_{K}\qquad\text{and}\qquad\mathcal{D}^{\text{min}}%
			(E_{2,1728}(1))=\left(  \sqrt[6]{2}\right)  ^{24}\mathcal{O}_{K}.
			\]
		\end{example}
		\subsection{Non-singular $j$-invariants}
		
		In this section we examine which $t_0$ give $j_{p, 1}(t_0) = j_{p, 2}(t_0)$ and over which fields the curves $\mathcal{C}_{p, i}(t_0)$, parameterized by such $t_0$, are isogenous versus isomorphic.
		In particular, we can additionally assume that $j_{p, 1}(t_0) \neq 0, 1728$.
		Finding such $t_0$ is as simple as finding the roots of $t^p(j_{p, 1}(t) - j_{p, 2}(t)) = 0$.
		For each root $t_0$ of $t^p(j_{p,1}(t) - j_{p, 2}(t))$ there are three possibilities for $\mathcal{C}_{p, 1}(t_0)$ and $\mathcal{C}_{p, 2}(t_0)$.  
		First, such $t_0$ could produce singular curves.  Such curves will have $j$-invariants $0$ or $1728$ and discussed in the section above.  
		Next, if $t_0$ is only defined over number fields, $K$, such that the CM endomorphisms of $\mathcal{C}_{p, i}(t_0)$ are already defined over $K$, 
		then $\mathcal{C}_{p, 1}(t_0)$ and $\mathcal{C}_{p, 2}(t_0)$ will be isomorphic.
		The final possibility is that $t_0$ is defined over a field $K$ not containing the CM endomorphisms of $\mathcal{C}_{p, i}(t_0)$.
		It is in this final case that we must check for discriminant ideal twins.
		
		\begin{theorem}
			\label{Thm:equaljinv}
			If $E_1, E_2$ are $p$-isogenous discriminant ideal twins defined over a number field $K$ with $j = j(E_1) = j(E_2)$, but $j \neq 0$ or $1728$, then $j$ is as follows.
			For $p = 5, 13$, $t_0 \in K$ and
			\begin{equation}
				j =
				\begin{cases}
					56576t_0 + 632000 & \text{if}\ t_0 \text{ is a root of } t^2 - 125, \\
					956448000t_0 + 3448440000 & \text{if}\ t_0 \text{ is a root of } t^2 - 13. \\
				\end{cases}
			\end{equation}
			For $p = 2, 3, 7$ we require that $p\mathcal{O}_K$ is a square but $\Q(\sqrt{-p})$ is not a subfield of $K$ and
			\begin{equation}
				j = 
				\begin{cases}
					8000 & p = 2 \text{ and } t_0 = 64, \\
					54000 & p = 3 \text{ and } t_0 = 27, \\
					-3375 & p = 7 \text{ and } t_0 = -7, \\
					16581375 & p = 7 \text{ and } t_0 = 7. \\
				\end{cases}
			\end{equation}
		\end{theorem}
		
		\begin{proof}
			We first find the values of $t_0$ so that $j_{p, 1}(t_0) = j_{p,2}(t_0)$.  Then, in the third of the cases listed above, we classify when these $t_0$ give discriminant ideal twins.  First, go through the cases in order.
			
			The singular $t_0$ values are given by \cite{Bariso} in Lemma 3.3 and are roots of the following factors of $t^p(j_{p,1}(t) - j_{p, 2}(t))$ given in Table \ref{Tab:singt0}.
			In particular, these are the roots that correspond to $j$-invariant $0$ or $1728$ and thus we discard them.
			
			\begin{center}
				\begin{table}[ht]
					\caption{Singular Factors}
					\begin{tabular}{cc} \toprule
						$p$ & factor \\ \hline
						$ 2 $  & $ t + 64 $ \\ \hline
						
						$ 3 $  & $ t + 27 $ \\ \hline
						
						$ 5 $  & $ t^{2} + 22 \, t + 125 $ \\ \hline
						
						$ 7 $  & $ t^{2} + 13 \, t + 49 $ \\ \hline
						
						$ 13 $  & $ t^{2} + 6 \, t + 13 $ \\ 
						$  $  & $ t^{2} + 5 \, t + 13 $ \\ \bottomrule
					\end{tabular}
					\label{Tab:singt0}
				\end{table}
			\end{center}
			
			To compute the factors corresponding to the second case where the parameter $t_0 \in K$ gives rise to curves $\mathcal{C}_{p, 1}(t_0)$ and $\mathcal{C}_{p, 2}(t_0)$ already isomorphic over $K$, we can simply check.
			Using each non-singular factor $f(t)$ of $t^p(j_{p,1}(t) - j_{p, 2}(t))$ we construct the number field $K = \mathbb{Q}(t_0) = \mathbb{Q}[t]/f(t)$ and the curves $\mathcal{C}_{p, 1}(t_0)$ and $\mathcal{C}_{p, 2}(t_0)$.
			From here we simply check if the curves are isomorphic over $K$, or equivalently, that $K$ contains the curves' CM endomorphism ring.
			The factors that give rise to isomorphic curves are listed below, up to Galois-action, i.e., one per factor instead of all $d$ where $d$ is the degree of $f(t)$:
			\begin{center}
				\renewcommand{\arraystretch}{1.3}
				\renewcommand{\arraycolsep}{.05cm}
				\begin{tabular}{ccc} \toprule
					$p$ & $j$-invariant & factor \\  \hline
					$ 2 $ & $ -3375 $ & $ t^{2} + 47 \, t + 4096 $ \\ \hline
					
					$ 3 $ & $ 8000 $ & $ t^{2} + 46 \, t + 729 $ \\
					$  $ & $ -32768 $ & $ t^{2} - 10 \, t + 729 $ \\ \hline
					
					$ 5 $ & $ -32768 $ & $ t^{2} + 18 \, t + 125 $ \\
					$  $ & $ 287496 $ & $ t^{2} + 4 \, t + 125 $ \\
					$  $ & $ -884736 $ & $ t^{2} - 14 \, t + 125 $ \\ \hline
					$ 7 $ & $ -\frac{284544}{49} t_{0}^{3} - \frac{2845440}{49} t_{0}^{2} - \frac{569088}{49} t_{0} + 994752 $ & $ t^{4} + 10 \, t^{3} + 51 \, t^{2} + 490 \, t + 2401 $ \\
					$  $ & $ 54000 $ & $ t^{2} + 11 \, t + 49 $ \\
					$  $ & $ -884736 $ & $ t^{2} + 5 \, t + 49 $ \\
					$  $ & $ -12288000 $ & $ t^{2} - 11 \, t + 49 $ \\ \hline
					
					$ 13 $ & $ -\frac{22165248}{13} t_{0}^{3} - \frac{132991488}{13} t_{0}^{2} - \frac{221652480}{13} t_{0} + 10275264 $ & $ t^{4} + 6 \, t^{3} + 23 \, t^{2} + 78 \, t + 169 $ \\
					$  $ & $ \frac{447971328}{13} t_{0}^{3} - \frac{447971328}{13} t_{0}^{2} - \frac{11199283200}{13} t_{0} - 2994536448 $ & $ t^{4} - t^{3} - 12 \, t^{2} - 13 \, t + 169 $ \\
					$  $ & $ -\frac{272875500}{13} t_{0}^{3} + \frac{3820257000}{13} t_{0} + 1417905000 $ & $ t^{4} - t^{2} + 169 $ \\
					$  $ & $ 54000 $ & $ t^{2} + 7 \, t + 13 $ \\
					$  $ & $ 287496 $ & $ t^{2} + 4 \, t + 13 $ \\
					$  $ & $ -12288000 $ & $ t^{2} + 2 \, t + 13 $ \\
					$  $ & $ -884736000 $ & $ t^{2} - 3 \, t + 13 $ \\ \bottomrule
				\end{tabular}
			\end{center}
			
			What is left are the $p$-isogenous, non-$K$-isomorphic curves with the same $j$-invariant.  Several factors give curves that are already discriminant ideal twins:
			\begin{center}
				\begin{tabular}{ccc} \toprule
					$p$ & $j$-invariant & factor \\  \hline
					$ 5 $ & $ 56576 t_{0} + 632000 $ & $ t^{2} - 125 $ \\ \hline
					$ 13 $ & $ 956448000 t_{0} + 3448440000 $ & $ t^{2} - 13 $ \\ \bottomrule
				\end{tabular}
			\end{center}
			
			For the $p$-isogenous, non-$K$-isomorphic curves with the same $j$-invariant and different minimal discriminant ideals, 
			we see that their minimal discriminant ideals vary by a factor of 6 at the prime above $p$.
			\begin{center}
				\begin{tabular}{ccccc}
					$p$ & $j$-invariant & factor & $\mathfrak{D}^{min}_1/\mathfrak{D}^{min}_2$ & CM Field\\ \toprule
					$ 2 $ & $ 8000 $ & $ t - 64 $ & $ (2\mathcal{O}_K)^{-6} $ & $\Q(\sqrt{-2})$\\ \hline
					
					$ 3 $ & $ 54000 $ & $ t - 27 $ & $ (3 \mathcal{O}_K)^{-6}$ & $\Q(\sqrt{-3})$ \\ \hline
					
					$ 7 $ & $ -3375 $ & $ t + 7 $ & $ (7 \mathcal{O}_K)^6$  & $\Q(\sqrt{-7})$ \\ \hline
					$ 7 $ & $ 16581375 $ & $ t - 7 $ & $ (7\mathcal{O}_K)^6$ & $\Q(\sqrt{-7})$ \\ \bottomrule
				\end{tabular}
			\end{center}
			As they have CM by an order in $\Q(\sqrt{-p})$, if $K$ contains $\Q(\sqrt{-p})$, they are isomorphic.
			Thus, these curves can only lead to discriminant ideal twins in fields that do not contain $\Q(\sqrt{-p}).$
			For them to be discriminant ideal twins, we need their discriminant valuations to all be congruent modulo $12$.  As they currently differ only at primes above $p$ and only by $6e$th powers where $\nup(p) = e$, they can only potentially be discriminant twins in fields where $e$ is even for all primes $\fp$ dividing $p\mathcal{O}_K$, i.e., when $p\mathcal{O}_K$ is a square.  
		\end{proof}
		
		Note that we only proved one direction, that if $E_1, E_2$ are discriminant ideal twins with the same $j$-invariant, then the $j$-invariant is one from a small list.
		To say anything in the other direction, we will need results from the next section.

		\section{General  classification of \texorpdfstring{$p$}{p}-isogenous discriminant ideal twins}\label{sec:pfmainthm}
		
		In this section, we will prove our main result, which classifies discriminant ideal twins over number fields. In particular, we will find explicit conditions on $t$ for which the elliptic curve $\Ci$ are discriminant (ideal) twins. As Lemma~\ref{lemmapotgoodred} showed that $p$-isogenous  discriminant ideal twins must have everywhere potentially good reduction, we start by determining conditions on $t$ for which the elliptic curves $\Ci$ have potentially good reduction.
		
		\begin{lemma}
			\label{lem:potentiallygood}Let $K$ be a number field and let $p\in\left\{
			2,3,5,7,13\right\}  $. Suppose further that $E_{1}$ and $E_{2}$ are
			$p$-isogenous elliptic curves over $K$ such that the $j$-invariants of $E_{1}$
			and $E_{2}$ are not both identically $0$ or $1728$. Then $E_{1}$ and $E_{2}$
			have potentially good reduction at a prime $\mathfrak{p}$ of $K$ if and only
			if there are $t_{0}\in K$ and $d_{0}\in\mathcal{O}_{K}$ with $0\leq\nu_{\mathfrak{p}%
			}(t_{0})\leq\frac{12\nu_{\mathfrak{p}}(p)}{p-1}$ such that $E_{i}$ is
			$K$-isomorphic to $\mathcal{C}_{p,i}(t_{0},d_{0})$.
		\end{lemma}
		
		\begin{proof}
			By Theorem \ref{parafamilies}, there are $t_{0}\in K$ and $d_{0}\in\mathcal{O}_{K}$ such that $E_{i}$ is $K$-isomorphic to $\mathcal{C}_{p,i}\!\left(  t_{0},d_{0}\right)  $. For a prime $\fp$ of $K$, it is the case that $\nup(j(\mathcal{C}_{p,1}\!\left(  t_{0},d_{0}\right) ))\geq 0$ if and only if $\nup(j(\mathcal{C}_{p,2}\!\left(  t_{0},d_{0}\right) ))\geq0$. 
			As a result, we work with the $j$-invariant of $\mathcal{C}_{p,2}\!\left(  t_{0},d_{0}\right)  $ as it has smaller coefficients and smaller exponents in the denominator. We note that $j\!\left(  \mathcal{C}_{p,i}\!\left(  t_{0},d_{0}\right)  \right)  =j_{p,i}\!\left(  t_{0}\right)$, where $j_{p,i}\!\left(  t_{0}\right)  $ is as given in Table~\ref{ta:invariants}. From Table \ref{ta:invariants}, we have that $\nu_{\mathfrak{p}}\!\left(  j_{p,2}\!\left(  t_{0}\right)  \right)  $ is as follows:
			\begin{equation}
				\scalebox{1}{
					\renewcommand{\arraystretch}{1.18}
					\renewcommand{\arraycolsep}{.3cm}
					$\begin{array}
						[c]{cc}\bottomrule
						p & \nu_{\mathfrak{p}}(j_{p,2}(t_{0}))\\\hline
						2 & 3\nu_{\mathfrak{p}}(t_{0}+16)-\nu_{\mathfrak{p}}(t_{0})\\\hline
						3 & 3\nu_{\mathfrak{p}}(t_{0}+3)+\nu_{\mathfrak{p}}(t_{0}+27)-\nu
						_{\mathfrak{p}}(t_{0})\\\hline
						5 & 3\nu_{\mathfrak{p}}(t_{0}^{2}+10t_{0}+5)-\nu_{\mathfrak{p}}(t_{0})\\\hline
						7 & 3\nu_{\mathfrak{p}}(t_{0}^{2}+5t_{0}+1)+\nu_{\mathfrak{p}}(t_{0}%
						^{2}+13t_{0}+49)-\nu_{\mathfrak{p}}(t_{0})\\\hline
						13 & 3\nu_{\mathfrak{p}}(t_{0}^{4}+7t_{0}^{3}+20t_{0}^{2}+19t_{0}%
						+1)+\nu_{\mathfrak{p}}(t_{0}^{2}+5t_{0}+13)-\nu_{\mathfrak{p}}(t_{0})\\
						\bottomrule
					\end{array}$}
				\label{valjpoly}
			\end{equation}
			
			First, assume that $E_{i}$ has potentially good reduction at $\mathfrak{p}$. We claim that $0 \leq \nup(t_0) \leq \frac{12 \nup(p)}{p-1}$. To get the lower bound, we proceed by contradiction. Suppose $\nu_{\mathfrak{p}}(t_{0})=-k$ for some positive integer $k$. Observe that for any monic polynomial $f(t)\in\mathcal{O}_{K}[t]$, $\nu_{\mathfrak{p}}(f(t_{0}))=-\deg(f(t_{0}))k$. Applying this to (\ref{valjpoly}) yields:
			\[%
			\scalebox{1}{
				\renewcommand{\arraystretch}{1.18}
				\renewcommand{\arraycolsep}{.3cm}
				$\begin{array}
					[c]{cccccc}\bottomrule
					p & 2 & 3 & 5 & 7 & 13\\\hline
					\nu_{\mathfrak{p}}(j_{p,2}(t_{0})) & -3k+k & -3k-k+k & -6k+k & -6k-2k+k &
					-12k-2k+k\\\hline
				\end{array}$}
			\]
			Hence for all $p \in \{2, 3, 5, 7, 13\}$ we get $\nu_{\mathfrak{p}}(j_{p,2}(t_{0}))=-pk<0$, which is our desired contradiction,
			as if $E_{i}$ has potentially good
			reduction at $\mathfrak{p}$, 
			then we have by Proposition \ref{prop:integralj} that
			$\nu_{\mathfrak{p}}\!\left(  j\!\left(  \mathcal{C}_{p,i}\!\left(  t_{0}%
			,d_{0}\right)  \right)  \right)  \geq0$. 
			
			We now establish the claimed upper bound. 
			So suppose that $\nu_{\mathfrak{p}}(t_{0})=k$ for $k$ a non-negative integer, and set $\nu_{\mathfrak{p}}(p)=e$. Then we have that $\nu_{\mathfrak{p}}(j_{p,2}(t_{0}))\geq\lambda_{p}$, where $\lambda_{p}$ is as given in (\ref{valj}):
			\begin{equation}%
				\scalebox{1}{
					\renewcommand{\arraystretch}{1.18}
					\renewcommand{\arraycolsep}{.4cm}
					$\begin{array}
						[c]{cc}\bottomrule
						p & \lambda_{p}\\\hline
						2 & 3\min(k,4e)-k\\\hline
						3 & 3\min(k,e)+\min(k,3e)-k\\\hline
						5 & 3\min(2k,e)-k\\\hline
						7 & \min(k,2e)-k\\\hline
						13 & \min(k,e)-k\\
						\bottomrule
					\end{array}$}
				\label{valj}%
			\end{equation}
			Note that $\nu_{\mathfrak{p}}(j_{p,2}(t_{0}))=\lambda_{p}$, except possibly, when the quantities appearing inside the minimum are equal. Now suppose that $\nu_{\mathfrak{p}}(j_{p,2}(t_{0}))>\lambda_{p}$. Then, the quantities appearing inside of the minimum of $\lambda_{p}$ are equal. In other words, we have that $k$ is equal to one of the values given in~\eqref{valuesofksec4}.
			\begin{equation}
				\scalebox{1}{
					\renewcommand{\arraystretch}{1.18}
					\renewcommand{\arraycolsep}{.3cm}
					$\begin{array}
						[c]{cccccc}\hline
						p & 2 & 3 & 5 & 7 & 13\\\hline
						k & 4e & e\text{ or }3e & \frac{e}{2} & 2e & e\\\hline
					\end{array}$}
				\label{valuesofksec4}
			\end{equation}
			In each case, we have that $k=\nu_{\mathfrak{p}}(t_{0})\leq\frac{12e}{p-1}$. Note also that in each case, $k=0$ if and only if $e=0$.
			
			Next, suppose that $\nu_{\mathfrak{p}}(j_{p,2}(t_{0}))=\lambda_{p}$. Then we
			have that $\lambda_{p}$ takes on one of the following values gives in~\eqref{valsoflampsec4}:%
			
			{\begin{equation}
					\renewcommand{\arraystretch}{1.2}
					\begin{tabular}
						[c]{cclcclccl}\hline
						$p$ & \multicolumn{8}{c}{$\text{Possible values of }\lambda_{p}$}\\\hline
						$2$ & $2k$ & $\text{ if }k\leq4e$ & $\quad$ & & & & $12e-k$ & $\text{ if }4e\leq k$ \\\hline
						$3$ & $3k$ & $\text{ if }k\leq e$ &  & $3e$ & $\text{ if }e\leq k\leq3e$ &  &
						$4e-k$ & $\text{ if }3e\leq k$\\\hline
						$5$ & $5k$ & $\text{ if }2k\leq e$ &  & & & & $3e-k$ & $\text{ if }e\leq2k$ \\\hline
						$7$ & $0$ & $\text{ if }k\leq2e$ &  & & & & $2e-k$ & $\text{ if }2e\leq k$ \\\hline
						$13$ & $0$ & $\text{ if }k\leq e$ &  & & & & $e-k$ & $\text{ if }2\leq k$ \\\hline
					\end{tabular}
					\label{valsoflampsec4}%
			\end{equation}}
			
			Since $E_{i}$ has potentially good reduction at $\mathfrak{p}$, we have that $\nu_{\mathfrak{p}}(j_{p,2}(t_{0}))=\lambda_{p}\geq0$. This inequality, together with \eqref{valsoflampsec4}, now shows via a case-by-case analysis that $k=\nu_{\mathfrak{p}}(t_{0})\leq\frac{12e}{p-1}$, and, in addition, $k=0$ if and only if $e=0$.
			
			In sum, we have established that if $E_{i}$ has potentially good reduction at $\mathfrak{p}$, then $0 \leq \nup(t_0) \leq \frac{12 \nup(p)}{p-1}$. The converse follows since the above shows that if $0 \leq \nup(t_0) \leq \frac{12 \nup(p)}{p-1}$ for some prime~$\mathfrak{p}$, then $\nu_\fp(j_{p,2}(t_0))\ge0$. Thus, $E_2$ has potentially good reduction at $\mathfrak{p}$. Since this is an isogeny invariant, it follows that $E_1$ also has potentially good reduction at $\mathfrak{p}$. This establishes the lemma, and we note that $t_{0}$ is a unit, or $t_{0}$ is only divisible by primes that divide $p$, and the valuation at these primes is bounded.
		\end{proof}
		
		Now that we have established a characterization about two $p$-isogenous elliptic curves having potentially good reduction in terms of the parametrized families $\Ci$, we are ready to begin our explicit classification of
		prime isogenous discriminant ideal twins. So suppose that $E_1$ and $E_2$ are $p$-isogenous discriminant ideal twins such that the $j$-invariants of $E_1$ and $E_2$ are not both identically $0$ or $1728$.
		The $j$-invariants can, however, be the same, just not $0$ or $1728$.
		By Lemma~\ref{lemmapotgoodred}, they have everywhere potentially good reduction. By Lemma
		\ref{lem:potentiallygood}, there are $t_{0},d_{0}\in\mathcal{O}_{K}$ with
		$0\leq\nu_{\mathfrak{p}}(t_{0})\leq\frac{12\nu_{\mathfrak{p}}(p)}{p-1}$ such
		that $E_{i}$ is $K$-isomorphic to $\mathcal{C}_{p,i}(t_{0},d_{0})$. We now show
		that the assumption that $E_{1}$ and $E_{2}$ are discriminant ideal twins
		gives a further restriction on $t_{0}$, namely that it must be a multiple of
		$\frac{12}{p-1}$. To state our result, we recall the following notation:
		if $r$ is a positive integer and $\mathcal{O}$ is a ring, then $\mathcal{O}%
		^{r}=\left\{  a^{r}:a\in\mathcal{O}\right\}  $.
		
		\begin{theorem}
			\label{PropMinDisc}Let $K$ be a number field and let $p\in\left\{
			2,3,5,7,13\right\}  $. Suppose that $E_{1}$ and $E_{2}$ are elliptic curves
			over $K$ such that their $j$-invariants are not both identically $0$ or
			$1728$. If $E_{1}$ and $E_{2}$ are $p$-isogenous discriminant ideal twins,
			then there exist $t_{0},d_{0}\in\mathcal{O}_{K}$ such that $E_{i}$ is
			$K$-isomorphic to $\mathcal{C}_{p,i}(t_{0},d_{0})$, and for each prime
			$\mathfrak{p}$ of $K$,%
			\[
			\nu_{\mathfrak{p}}(t_{0})=\frac{12}{p-1}k,\qquad0\leq k\leq\nu_{\mathfrak{p}%
			}(p).
			\]
			If further $E_{1}$ and $E_{2}$ are $p$-isogenous discriminant twins, then $t_{0}\in\mathcal{O}_K^{\frac{12}{p-1}}$.
		\end{theorem}
		
		\begin{proof}
			Since $E_{1}$ and $E_{2}$ are discriminant ideal twins, we have by Lemma
			\ref{lemmapotgoodred} that they have everywhere potentially good reduction.
			The additional assumption that $E_{1}$ and $E_{2}$ are $p$-isogenous gives us
			by Lemma \ref{lem:potentiallygood} that there are $t_{0}\in K$ and $d_{0}\in
			\mathcal{O}_{K}$ such that $E_{i}$ is $K$-isomorphic to $\mathcal{C}%
			_{p,i}(t_{0},d_{0})$. Moreover, since $0\leq\nu_{\mathfrak{p}}(t_{0})\leq\frac
			{12\nu_{\mathfrak{p}}(p)}{p-1}$ for each prime $\mathfrak{p}$ we have that $t_0 \in \OK$. In particular, $\nu_{\mathfrak{p}}(t_{0})=0$ for each prime
			$\mathfrak{p}\nmid p$.
			
			Without loss of generality, we may take $E_{i}=\mathcal{C}_{p,i}(t_{0},d_{0}%
			)$, and thus $E_{i}$ is given by an integral Weierstrass model. Now let
			$\mathfrak{D}_{i}^{\text{min}}$ denote the minimal discriminant ideal of
			$E_{i}$. Then $\mathfrak{D}_{1}^{\text{min}}=\mathfrak{D}_{2}^{\text{min}}$
			since $E_{1}$ and $E_{2}$ are discriminant ideal twins.
			
			For any prime $\mathfrak{p}$, we take $E_{i,\mathfrak{p}}$ to be an $\fp$-minimal model of $E_{i}$ at $\mathfrak{p}$. Suppose further that $E_{i,\fp}$  is given by an integral Weierstrass model.
			Consequently, $\nu_{\mathfrak{p}%
			}(\Delta(E_{1,\mathfrak{p}}))=\nu_{\mathfrak{p}}(\Delta(E_{2,\mathfrak{p}}))$.
			In particular, there exists $\mu_{\mathfrak{p}}\in\mathcal{O}_{K}$ such that
			$\nu_{\mathfrak{p}}(\mu_{\mathfrak{p}})=0$ and $\Delta(E_{1,\mathfrak{p}}%
			)=\mu_{\mathfrak{p}}\Delta(E_{2,\mathfrak{p}})$. Note that if $E_{1}$ and
			$E_{2}$ are discriminant twins, then we may suppose that $\mu_\mathfrak{p}=1$. 
			From Section~\ref{sec:mindiscs}, there exists $u_{i,\mathfrak{p}}%
			\in\mathcal{O}_{K}$ such that $u_{i,\mathfrak{p}}^{-12}\Delta(E_{i})
			=\Delta(E_{i,\mathfrak{p}})$. Consequently, $u_{1,\mathfrak{p}
			}^{-12}\Delta(E_{1})=\mu_{\mathfrak{p}}u_{2,\mathfrak{p}}^{-12}\Delta(E_{2})$. By
			(\ref{eq:Delatap12}),
			\begin{equation}
				\frac{t_{0}^{p-1}}{p^{12}}=\frac{\Delta(E_{1})}{\Delta(E_{2})}=\mu
				_{\mathfrak{p}}\frac{u_{1,\mathfrak{p}}^{12}}{u_{2,\mathfrak{p}}^{12}}%
				\qquad\Longrightarrow\qquad t_{0}^{p-1}=\mu_{\mathfrak{p}}\left(
				p\frac{u_{1,\mathfrak{p}}}{u_{2,\mathfrak{p}}}\right)  ^{12}.\label{mueqn}
			\end{equation}
			In particular, for each $\mathfrak{p}|p$, we have that $\nu_{\mathfrak{p}}(t_{0})=\frac{12}{p-1}k$ for some non-negative integer $k$. Since $\nu_{\mathfrak{p}}(t_{0})=0$ for each prime $\mathfrak{p}%
			\nmid p$, we conclude that for each prime $\mathfrak{p}$ of $K$,%
			\[
			\nu_{\mathfrak{p}}(t_{0})=\frac{12}{p-1}k
			\]
			for some $k$ satisfying $0\leq k\leq\nu_{\mathfrak{p}}(p)$. We note that the upper bound is a consequence of Lemma~\ref{lem:potentiallygood}.
			
			Lastly, suppose that $E_{1}$ and $E_{2}$ are discriminant twins. Then for each
			prime $\mathfrak{p}$ of $K$, we can take $\mu_{\mathfrak{p}}=1$ in
			(\ref{mueqn}). Consequently, we have that for each $\mathfrak{p}$, we can
			write%
			\begin{equation}
				t_{0}=\left(  p\frac{u_{1,\mathfrak{p}}}{u_{2,\mathfrak{p}}}\right)
				^{\frac{12}{p-1}}.\label{eqnpower}%
			\end{equation}
			Now consider the inclusion $K\hookrightarrow K_{\mathfrak{p}}$. From
			(\ref{eqnpower}), $t_{0}\in\mathcal{O}_{K_{\mathfrak{p}}}^{\frac{12}{p-1}}$.
			We conclude by the Grunwald-Wang
			Theorem \cite{MR33801} that $t_{0}\in\mathcal{O}_K^{\frac{12}{p-1}}$.
		\end{proof}

		The converse to both Theorems \ref{PropMinDisc} and \ref{Thm:equaljinv} is true for $p\neq2$, and
		constitutes our next result, Theorem \ref{ThmClassification}.
		Specifically, for $p\in\left\{  3,5,7,13\right\}
		$, the following theorem explicitly classifies all $p$-isogenous elliptic
		curves that are discriminant ideal twins and whose $j$-invariants are not both
		identically $0$ or $1728$:
		
			

		\begin{theorem}
			\label{ThmClassification}
			Let $K$ be a number field and let $p\in\left\{
			3,5,7,13\right\}  $. Suppose that $E_{1}$ and $E_{2}$ are elliptic curves over
			$K$ such that their $j$-invariants are not the same.
			Then $E_{1}$ and $E_{2}$ are $p$-isogenous discriminant ideal twins if and
			only if there exist $t_{0},d_{0}\in\mathcal{O}_{K}$ such that $E_{i}$ is
			$K$-isomorphic to the elliptic curve $\mathcal{C}_{p,i}(t_{0},d_{0})$, and for
			each prime $\mathfrak{p}$ of $K$,%
			\begin{equation}
				\nu_{\mathfrak{p}}(t_{0})=\frac{12}{p-1}k,\qquad0\leq k\leq\nu_{\mathfrak{p}%
				}(p),\label{thmclasseqn}%
			\end{equation}
			and $\operatorname*{typ}_{\mathfrak{p}}(E_{1})=\operatorname*{typ}%
			_{\mathfrak{p}}(E_{2})$.
			Furthermore, $E_{1}$ and $E_{2}$ are discriminant twins if and only if
			$t_{0}\in\mathcal{O}_{K}^{\frac{12}{p-1}}$ and $t_0$ satisfies \eqref{thmclasseqn} for each prime $\fp$ of $K$.
			
			If $E_{1}$ and $E_{2}$ are elliptic curves over
			$K$ such that their $j$-invariants are the same, but not $0$ or $1728$,
			then $E_1$ and $E_2$ have CM by some order $\mathcal{O}$ and the above statement holds so long as $\mathcal{O}$ is not contained in $K$.
		\end{theorem}
		
		\begin{proof}
			We note that most of the forward direction is a consequence of Theorem
			\ref{PropMinDisc}. In particular, it remains to show that $\operatorname*{typ}%
			_{\mathfrak{p}}(E_{1})=\operatorname*{typ}_{\mathfrak{p}}(E_{2})$ for each
			prime $\mathfrak{p}$ of $K$. This will follow as a consequence of the
			converse, which we consider below.
			
			Let $t_{0},d_{0}\in\mathcal{O}_{K}$ such that $E_{i}=\mathcal{C}_{p,i}%
			(t_{0},d_{0})$ is an elliptic curve, and suppose further that for each prime
			$\mathfrak{p}$ of $K$, (\ref{thmclasseqn}) holds. By Lemma
			\ref{lem:potentiallygood}, we deduce that $E_{1}$ and $E_{2}$ have everywhere
			potentially good reduction. Now let $\mathfrak{D}_{i}^{\text{min}}$ denote the
			minimal discriminant of $E_{i}$. By Corollary \ref{DokLemma}, we have that for
			each prime $\mathfrak{p}\nmid p$, it is the case that $\nu_{\mathfrak{p}%
			}(\mathfrak{D}_{1}^{\text{min}})=\nu_{\mathfrak{p}}(\mathfrak{D}%
			_{2}^{\text{min}})$ and $\operatorname*{typ}_{\mathfrak{p}}(E_{1}%
			)=\operatorname*{typ}_{\mathfrak{p}}(E_{2})$. So it remains to consider the
			case when $\mathfrak{p}$ is a prime of $K$ such that $\mathfrak{p}|p$. 
			
			We now show that if $\mathfrak{p}|p$, then $\nu_{\mathfrak{p}}(\Delta
			(E_{1}))\equiv\nu_{\mathfrak{p}}(\Delta(E_{2}))\ \operatorname{mod}12$. This
			will allow us to invoke Lemmas \ref{KNforpneq3} and \ref{KNforpeq3}, which
			will then allow us to conclude that $E_{1}$ and $E_{2}$ are discriminant ideal
			twins, and for each prime $\mathfrak{p}$ of $K$, $\operatorname*{typ}%
			_{\mathfrak{p}}(E_{1})=\operatorname*{typ}_{\mathfrak{p}}(E_{2})$. 
			Since
			$\nu_{\mathfrak{p}}(t_{0})=\frac{12}{p-1}k$, we have that%
			\[
			\scalebox{1}{
				\renewcommand{\arraystretch}{1.05}
				\renewcommand{\arraycolsep}{.3cm}
				$\begin{array}
					[c]{ccccc}\bottomrule
					p & 3 & 5 & 7 & 13\\\hline
					\nu_{\mathfrak{p}}(t_{0}) & 6k & 3k & 2k & k\\\hline
				\end{array}$}
			\]
			with $k\leq\nu_{\mathfrak{p}}(p)$. For each $p$, we then have by inspection of
			Table \ref{ta:invariants} that $\nu_{\mathfrak{p}}(\Delta(E_{i}))$ is given
			by:
			\begin{equation}%
				\scalebox{1}{
					\renewcommand{\arraystretch}{1.18}
					\renewcommand{\arraycolsep}{.3cm}
					$\begin{array}
						[c]{ccc}\bottomrule
						p & i & \nu_{\mathfrak{p}}(\Delta(E_{i}))\\\hline
						3 & 1 & 18k + 2\nup(27 + t_0) + 6\nup(d_0)\\\cmidrule(lr){2-3}
						& 2 & 6k + 2\nup(27 + t_0) + 6\nup(d_0) \\\hline
						5 & 1 & 15k + 3\nup(125 + 22t_0 + t_0^2) + 6\nup(d_0) \\\cmidrule(lr){2-3}
						& 2 & 3k + 3\nup(125 + 22t_0 + t_0^2) + 6\nup(d_0) \\\hline
						7 & 1 & 14k + 2\nup(49 + 13t_0 + t_0^2) + 6\nup(d_0)\\\cmidrule(lr){2-3}
						& 2 &  2k + 2\nup(49 + 13t_0 + t_0^2)  + 6\nup(d_0)\\\hline
						13 & 1 & 13k + \nup((13 + 5t_0 + t_0^2)(13 + 6t_0 + t_0^2))  + 6\nup(d_0) \\\cmidrule(lr){2-3}
						& 2 & k + \nup((13 + 5t_0 + t_0^2)(13 + 6t_0 + t_0^2))  + 6\nup(d_0)\\
						\bottomrule
					\end{array}$}
				\label{valuations}%
			\end{equation}
			Observe that for each $p$, we have that $\nu_{\mathfrak{p}}(\Delta(E_{1}%
			))-\nu_{\mathfrak{p}}(\Delta(E_{2}))=12k$. Therefore $\nu_{\mathfrak{p}%
			}(\Delta(E_{1}))\equiv\nu_{\mathfrak{p}}(\Delta(E_{2}))\ \operatorname{mod}%
			12$. By Lemmas~\ref{KNforpneq3} and~\ref{KNforpeq3}, we conclude that
			$\nu_{\mathfrak{p}}(\mathfrak{D}_{1}^{\text{min}})=\nu_{\mathfrak{p}%
			}(\mathfrak{D}_{2}^{\text{min}})$ and $\operatorname*{typ}_{\mathfrak{p}%
			}(E_{1})=\operatorname*{typ}_{\mathfrak{p}}(E_{2})$ if $\mathfrak{p}\nmid2$. Therefore $E_{1}$ and $E_{2}$ are discriminant ideal
			twins and $\operatorname*{typ}_{\mathfrak{p}}(E_{1})=\operatorname*{typ}%
			_{\mathfrak{p}}(E_{2})$ for each prime $\mathfrak{p}$ of $K$.
			
			It remains to show that if $t_{0}\in\mathcal{O}_{K}^{\frac{12}{p-1}}$, then
			$E_{1}$ and $E_{2}$ are discriminant twins. 
			To this end, for a prime
			$\mathfrak{p}$ of $K$, 
			let $E_{i,\mathfrak{p}}$ be a $\mathfrak{p}$-minimal model of
			$E_{i}$. 
			By Lemma~\ref{newSilverman}, there
			exists $u_{i,\mathfrak{p}}\in\mathcal{O}_{K}$ such that $u_{i,\mathfrak{p}}^{-12}\Delta(E_{i})
			=\Delta(E_{i,\mathfrak{p}})$. Moreover, since $E_{1}$
			and $E_{2}$ are discriminant ideal twins it is the case that there is a
			$\mu_{\mathfrak{p}}\in\mathcal{O}_{K}$ such that $\nu_{\mathfrak{p}}
			(\mu_{\mathfrak{p}})=0$ and $\Delta(E_{1,\mathfrak{p}})=\mu_{\mathfrak{p}
			}\Delta(E_{2,\mathfrak{p}})$. By (\ref{eq:Delatap12}),
			\[
			\frac{t_{0}^{p-1}}{p^{12}}=\frac{\Delta(E_{1})}{\Delta(E_{2})}=\mu
			_{\mathfrak{p}}\frac{u_{1,\mathfrak{p}}^{12}}{u_{2,\mathfrak{p}}^{12}}%
			\qquad\Longrightarrow\qquad\mu_{\mathfrak{p}}=t_{0}^{p-1}\left(
			\frac{u_{2,\mathfrak{p}}}{pu_{1,\mathfrak{p}}}\right)  ^{12}.
			\]
			By assumption, $t_{0}^{p-1}\in\mathcal{O}_{K}^{12}$ and thus $\mu
			_{\mathfrak{p}}\in\mathcal{O}_{K}^{12}$. In particular, there is a
			$\kappa_{\mathfrak{p}}\in\mathcal{O}_{K}$ such that $\kappa_{\mathfrak{p}%
			}^{12}=\mu_{\mathfrak{p}}$ and $\nu_{\mathfrak{p}}(\kappa_{\mathfrak{p}})=0$.
			Now let $E_{1,\mathfrak{p}}^{\prime}$ be the elliptic curve obtained from
			$E_{1,\mathfrak{p}}$ via the isomorphism $[\kappa_\fp,0,0,0]$. Then $E_{1,\mathfrak{p}}$ is an $\mathfrak{p}$-minimal
			model  with $\Delta(E_{1,\mathfrak{p}}^{\prime}%
			)=\kappa_{\mathfrak{p}}^{-12}\Delta(E_{1,\mathfrak{p}})$. Hence $\Delta
			(E_{1,\mathfrak{p}})=\mu_{\mathfrak{p}}\Delta(E_{1,\mathfrak{p}}^{\prime})$,
			which, in turn, yields that $\Delta(E_{1,\mathfrak{p}}^{\prime})=\Delta
			(E_{2,\mathfrak{p}})$. This shows that for each $\mathfrak{p}$, there are
			$\mathfrak{p}$-minimal models of $E_{1}$ and $E_{2}$ having equal
			discriminants. Therefore $E_{1}$ and $E_{2}$ are discriminant twins.
			
			Note that the above holds as long as $E_1$ and $E_2$ do not both have $j$-invariant $0$ or $1728$.
			The curves could, however, be isomorphic and thus not discriminant (ideal) twins.
			If the curves are isomorphic over $K$, then they have the same $j$-invariant.
			In Theorem \ref{Thm:equaljinv} we give the values of $t_0$ that produce equal $j$-invariants.
			These curves will be $K$-isomorphic if and only if their CM order is contained in $K$ by Lemma \ref{SameJCM}.
		\end{proof}
		
		\begin{remark}\label{constructivermk}
			Theorem \ref{ThmClassification} is in fact constructive. For example, if $K$ has class number 1, then let $u_0, u_1,...,u_n \in \mathcal{O}_K^{\times}$ generate the unit group and for each $\fp_j \mid p$, 
			with $p \in \{2, 3, 5, 7, 13\}$, let $\pi_j$ generate $\fp_j$. Then for any $t_0 = \prod u_i^{m_i} \prod \pi_j^{12k_j/(p-1)}$ for $0 \leq k_j \leq \nup(\fp_j)$ and $m_i \in \mathbb{Z}$, the curves $\mathcal{C}_{p, 1}(t_0, d_0)$ and $\mathcal{C}_{p, 2}(t_0, d_0)$ are $p$-isogenous discriminant ideal twins if they are not $K$-isomorphic. If further each $m_i = 12 m_i'/(p-1)$ for $m_i' \in \mathbb{Z}$, then  $\mathcal{C}_{p, 1}(t_0, d_0)$ and $\mathcal{C}_{p, 2}(t_0, d_0)$ are discriminant twins, again, as long as they are not $K$-isomorphic.
			For $p \neq 2$, the finitely many potentially $K$-isomorphic cases are given in Theorem \ref{Thm:equaljinv}.
		\end{remark}
		
		The following example shows that the if and only if statement in Theorem~\ref{ThmClassification} fails for $p=2$.
		\begin{example}
			\label{ex:p2}
			Let $K = \mathbb{Q}(\sqrt[5]{2})$ and $a = \sqrt[5]{2}$. 
			This number field has class number one and unit group isomorphic to $C_2 \times \mathbb{Z} \times \mathbb{Z}$ with generators $u_0 = -1$,  $u_1 = a^3 + a^2 - 1, u_2 = a -1$.
			Take $(a) = \fp \mid 2$, then $\nup(2) = 5$.
			Theorem~\ref{ThmClassification} refers to $t_0 \in \OK$ with $\nup(t_0) = \frac{12k}{p-1} = 12k$ for some integer $k$ satisfying $0\leq k \leq \nup(2) = 5$. 
			Then $t_0$ is of the form $t_0 = \pm u_1^{k_1} u_2^{k_2} a^{12k}$ for $k_1, k_2 \in \mathbb{Z}$ and $0 \leq k \leq 5$.
			
			For $k = 2$, take $t_0 = a^{24} = 16a^2$ and while the elliptic curves $C_{2, 1}(t_0, 1)$ and $C_{2, 2}(t_0, 1)$ have the same minimal discriminants, their Kodaira-N\'{e}ron type are ${\rm I}_4^*$ and ${\rm II}^*$ respectively. 
			
			For $k = 3$, take $t_0 = a^{36} = 128a$. Now these curves $C_{2, 1}(t_0, 1)$ and $C_{2, 2}(t_0, 1)$ have Kodaira-N\'{e}ron type ${\rm I}_{4}^*$ and ${\rm I}_{16}^*$, respectively, and discriminants  $3072a^4 + 1536a^3 + 256a^2 + 4096$ and $512a^3 + 768a^2 + 384a + 64$ respectively with valuation at $\fp$ given by $30$ and $42$. 
		\end{example}
		
		Building on Remark~\ref{constructivermk}, we have the following proposition.
		
		\begin{theorem} \label{InflyDITs}
			Let $K$ be a number field with an infinite unit group. Then, up to twist, there are infinitely many prime isogenous discriminant (ideal) twins for $p \in \{3, 5, 7, 13\}$.
		\end{theorem}
		
		\begin{proof}
			Let $r$ and $s$ denote the number of real and complex embeddings,
			respectively, of $K$ into $\mathbb{C}$. 
			Then $\left[  K:\mathbb{Q}\right]  =r+2s$, and by Dirichlet's Unit Theorem,
			$\mathcal{O}_{K}^{\times}\cong\mu\!\left(  K\right)  \times\mathbb{Z}^{r+s-1}%
			$. By assumption, $r+s>1$. In particular, $\mathcal{O}_{K}$ has a fundamental
			units $u_1,...,u_{r+s-1}$. 
			By Theorem~\ref{ThmClassification}, we have that for each product $t_0 = \prod_{i = 1}^{r+s-1} u_i^{m_i}$ where $m_i \in \mathbb{Z}$, 
			and for each $p \in \{3, 5, 7, 13\}$, the curves
			$\mathcal{C}_{p,1}(t_0,1)$ and $\mathcal{C}%
			_{p,2}(t_0,1)$ are $p$-isogenous discriminant ideal twins,
			except for the finitely many exceptions given in Theorem \ref{Thm:equaljinv}.
			For a fixed $t_0$, set
			\[
			j_{p,2}(t_0)=j\!\left(  \mathcal{C}_{p,2}(t_0,1)\right)  =\frac{f_p(t_0) }%
			{t_0}
			\]
			where $f_p(t_0)$ is the numerator of $j\!\left(  \mathcal{C}_{p,2}(t_0,1)\right)$ and a degree $p+1$  polynomial in $t_0.$
			Then $t_0$ is a root of the degree $p + 1$ polynomial%
			\begin{equation}\label{polyinf}
				f_p(X)  -j_{p, 2}(t_0) X.
			\end{equation}
			Thus \eqref{polyinf} is a degree $p+1$ map from $K$ to $K$.
			As $\mathcal{O}_K^{\times} \cong\mu\!\left(  K\right)  \times\mathbb{Z}^{r+s-1}$,
			\eqref{polyinf} restricted to $\mathcal{O}_K^{\times}$ still has infinite image.
			This shows that, up to twist, there are infinitely many $p$-isogenous discriminant ideal twins.
			
			Further, to get infinitely many discriminant twins, take $t_0\in (\OK^\times)^{\frac{12}{p-1}}$.
		\end{proof}

		\section{Explicit results over \texorpdfstring{$\Q$}{Q} and imaginary quadratic fields\label{sec:Examples}}
		
		In \cite{Deines2018}, Deines classified all semistable discriminant twins over
		$\mathbb{Q}$. The article also showed that isogeny classes of size two with at
		least one prime with multiplicative reduction could not have discriminant twins.
		Moreover, up to $j$-invariant, there were only finitely many such discriminant twins.
		This differs from the general case of number fields with an infinite unit group, as seen in Theorem~\ref{InflyDITs}.
		Motivated by these results and as an application of our main theorems, we now give a classification of all prime isogenous discriminant ideal twins over number fields with a finite unit group. 
		We begin with the case of rational elliptic curves.
		
		\begin{proposition}
			\label{classinQ}Up to twist, there are finitely many $p$-isogenous
			discriminant ideal twins over~$\mathbb{Q}$. These elliptic curves, $E_{1}$ and $E_{2}$, are given in Table~\ref{ta:clasoverQ}
			by their LMFDB label. The table also lists their $j$-invariant and minimal discriminant.
		\end{proposition}
		
		{\renewcommand*{\arraystretch}{1.24}
			\begin{longtable}[c]{ccccccc}
				\caption{Prime isogenous discriminant ideal twins over $\Q$}\\
				\hline
				$p$ & $E_{1}$ & $E_{2}$ & $j\!\left(  E_{1}\right)  $ & $j\!\left(
				E_{2}\right)  $ & $\Delta_{E_{1}}^{\text{min}}$ & $\Delta_{E_{2}}^{\text{min}}$ \\
				\hline
				\endfirsthead
				\caption[]{\emph{continued}}\\
				\hline
				$p$ & $E_{1}$ & $E_{2}$ & $j\!\left(  E_{1}\right)  $ & $j\!\left(
				E_{2}\right)  $ & $\Delta_{E_{1}}^{\text{min}}$ & $\Delta_{E_{2}}^{\text{min}}$\\
				\hline
				\endhead
				\hline
				\multicolumn{6}{r}{\emph{continued on next page}}
				\endfoot
				\hline
				\endlastfoot
				
				$2$ & $\href{http://www.lmfdb.org/EllipticCurve/Q/4225.h1/}{4225.h1}$ & $\href{http://www.lmfdb.org/EllipticCurve/Q/4225.h2/}{4225.h2}$ & $16974593$ & $4913$ & $274625$ & $274625$\\\cmidrule(lr){2-7}
				& $\href{http://www.lmfdb.org/EllipticCurve/Q/49.a3/}{49.a3}$ & $\href{http://www.lmfdb.org/EllipticCurve/Q/49.a4/}{49.a4}$ & $16581375$ & $-3375$ & $343$ & $-343$\\\hline
				$3$ & $\href{http://www.lmfdb.org/EllipticCurve/Q/196.a1/}{196.a1}$ & $\href{http://www.lmfdb.org/EllipticCurve/Q/196.a2/}{196.a2}$ & $406749952$ & $1792$ & $784$ & $784$\\\cmidrule(lr){2-7}
				& $\href{http://www.lmfdb.org/EllipticCurve/Q/676.b1/}{676.b1}$ & $\href{http://www.lmfdb.org/EllipticCurve/Q/676.b2/}{676.b2}$ & $-368484688$ & $-208$ & $-43264$ & $-43264$\\\hline
				$5$ & $\href{http://www.lmfdb.org/EllipticCurve/Q/1369.f1/}{1369.f1}$ & $\href{http://www.lmfdb.org/EllipticCurve/Q/1369.f2/}{1369.f2}$ & $38477541376$ & $4096$ & $50653$ & $50653$\\\cmidrule(lr){2-7}
				& $\href{http://www.lmfdb.org/EllipticCurve/Q/43264.f1/}{43264.f1}$ & $\href{http://www.lmfdb.org/EllipticCurve/Q/43264.f2/}{43264.f2}$ & $-23788477376$ & $64$ & $-1124864$ & $-1124864$\\\hline
				$7$ & $\href{http://www.lmfdb.org/EllipticCurve/Q/3969.f1/}{3969.f1}$ & $\href{http://www.lmfdb.org/EllipticCurve/Q/3969.f2/}{3969.f2}$ & $1168429123449$ & $21609$ & $3969$ & $3969$\\\cmidrule(lr){2-7}
				& $\href{http://www.lmfdb.org/EllipticCurve/Q/1369.b1/}{1369.b1}$ & $\href{http://www.lmfdb.org/EllipticCurve/Q/1369.b2/}{1369.b2}$ & $-371323264041$ & $999$ & $-1369$ & $-1369$\\\hline
				$13$ & $\href{http://www.lmfdb.org/EllipticCurve/Q/9025.a1/}{9025.a1}$ & $\href{http://www.lmfdb.org/EllipticCurve/Q/9025.a2/}{9025.a2}$ & $2045023375454208$ & $2101248$ & $45125$ &
				$45125$\\\cmidrule(lr){2-7}
				& $\href{http://www.lmfdb.org/EllipticCurve/Q/20736.n1/}{20736.n1}$ & $\href{http://www.lmfdb.org/EllipticCurve/Q/20736.n2/}{20736.n2}$ & $-39091613782464$ & $576$ & $-41472$ & $-41472$\\\hline
				$37$ & $\href{http://www.lmfdb.org/EllipticCurve/Q/1225.b1/}{1225.b1}$ & $\href{http://www.lmfdb.org/EllipticCurve/Q/1225.b2/}{1225.b2}$ & $-162677523113838677$ & $-9317$ & $-6125$ &
				$-6125$
				\label{ta:clasoverQ}	
		\end{longtable}}

		\begin{proof}
			Let $p$ be a prime and suppose that $X_{0}(p)$ has positive genus and a
			non-cuspidal $\mathbb{Q}$-rational point. Then $p \in \{11,17,19,37,43,67,163 \}$ \cite{MazurRatIsogPrimeDeg}. Up to quadratic twist, these non-cuspidal points determine
			$8$ distinct isogeny classes given below by their LMFDB label \cite[Lemma 5.1]{Bariso}.
			\begin{equation}\label{isogen1}
				\begin{tabular}
					[c]{cccccccc}\bottomrule
					$n$ & $11$& $17$ & $19$ & $37$ & $43$ & $67$ & $163$ \\\hline
					Isogeny class & \href{https://www.lmfdb.org/EllipticCurve/Q/121/a/}{121.a},\href{https://www.lmfdb.org/EllipticCurve/Q/121/b/}{121.b} & \href{https://www.lmfdb.org/EllipticCurve/Q/14450/b/}{14450.b} & \href{https://www.lmfdb.org/EllipticCurve/Q/361/a/}{361.a} & \href{https://www.lmfdb.org/EllipticCurve/Q/1225/b/}{1225.b} & \href{https://www.lmfdb.org/EllipticCurve/Q/1849/b/}{1849.b} & \href{https://www.lmfdb.org/EllipticCurve/Q/4489/b/}{4489.b} & \href{https://www.lmfdb.org/EllipticCurve/Q/26569/a/}{26569.a}\\\hline
				\end{tabular}
			\end{equation}
			
			The isogeny class \href{https://www.lmfdb.org/EllipticCurve/Q/1225/b/}{1225.b} consisting of $37$-isogenous elliptic curves has discriminant twins with j-invariants and discriminants as given in Table \ref{ta:clasoverQ}. 
			
			Next, the $17$-isogenous elliptic curves in the isogeny class \href{https://www.lmfdb.org/EllipticCurve/Q/14450/b/}{14450.b} have non-integral $j$-invariants, $-882216989/131072$ and $-297756989/2$. 
			Therefore, none of their quadratic twists will result in discriminant ideal twins by Lemma~\ref{lemmapotgoodred}. 
			
			We now claim that for  $p\in \{11,19,43,67,163\}$, there are no $p$-isogenous discriminant ideal twins.
			To this end, observe that the $j$-invariants of these elliptic curves are integral; see Table~\ref{ta:JDiscCM}. For these elliptic curves, the table also gives the $p$-adic valuation of their minimal discriminants. In particular, the $p$-adic valuations of the discriminants of the two elliptic curves are not congruent modulo $12$. Consequently, the $p$-adic valuations of the discriminant of each quadratic twist of these elliptic curves will also not be congruent modulo $12$. We conclude by Lemma~\ref{KNforpneq3} that there are no $p$-isogenous discriminant ideal twins.
			
			{\renewcommand*{\arraystretch}{1.18}
				\begin{longtable}[c]{cccccc}
					\caption{$j$-invariants and $p$-adic valuation of minimal discriminants of some $p$-isogenous elliptic curves }\\
					\hline
					Isogeny Class & $p$ & $j_1$ & $j_2$ & $\nu_p(\mathfrak{D}^{min}_1)$ & $\nu_p(\mathfrak{D}^{min}_2)$  \\ 
					\hline
					\endfirsthead
					\caption[]{\emph{continued}}\\
					\hline
					Isogeny Class & $p$ & $j_1$ & $j_2$ & $\nu_p(\mathfrak{D}^{min}_1)$ & $\nu_p(\mathfrak{D}^{min}_2)$  \\
					\hline
					\endhead
					\hline
					\multicolumn{4}{r}{\emph{continued on next page}}
					\endfoot
					\hline
					\endlastfoot
					
					\href{https://www.lmfdb.org/EllipticCurve/Q/121/a/}{121.a} & $11$ & $-121$ & $-24729001$ & $10$ & $2$ \\\hline
					\href{https://www.lmfdb.org/EllipticCurve/Q/121/b/}{121.b} & $11$ & $-32768$ & $-32768$ & $9$ & $3$ \\\hline
					\href{https://www.lmfdb.org/EllipticCurve/Q/361/a/}{361.a} & $19$ & $-884736$ & $-884736$ & $9$ & $3$ \\\hline
					\href{https://www.lmfdb.org/EllipticCurve/Q/1849/b/}{1849.b} & $43$ & $-884736000$ & $-884736000$ & $9$ & $3$ \\\hline
					\href{https://www.lmfdb.org/EllipticCurve/Q/4489/b/}{4489.b} & $67$ & $-147197952000$ & $-147197952000$ & $9$ & $3$ \\\hline
					\href{https://www.lmfdb.org/EllipticCurve/Q/26569/a/}{26569.a} & $163$ & $-262537412640768000$ & $-262537412640768000$ & $9$ & $3$
					\label{ta:JDiscCM}	
			\end{longtable}}

			It remains to consider the case when $X_{0}(p)$ has genus zero. First, suppose that
			$E_{1}$ and $E_{2}$ are $p$-isogenous elliptic curves over $\mathbb{Q}$ such that $j(E_{1})=j(E_{2})$. 
			By Theorems~\ref{thmforj-_1728} and \ref{Thm:equaljinv}, $E_1$ and $E_2$ are not discriminant ideal twins.

			Now suppose that $E_{1}$ and $E_{2}$ are $p$-isogenous discriminant ideal
			twins defined over $\mathbb{Q}$ such that $j(E_{1})\neq j(E_{2})$.
			By Theorem~\ref{PropMinDisc}, there are $t_{0},d_{0}\in\mathbb{Z}$ such that $E_{i}$ is $\mathbb{Q}$-isomorphic to $\mathcal{C}_{p,i}(t_{0},d_{0})$ with $t_0 = \pm p^{\frac{12k}{p-1}}$ where $k$ is either $0$ or $1$.
			Thus, $t_{0} \in\left\{
			\pm 1, \pm p^{\frac{12}{p-1}}\right\}  $. It is verified in 
			Remark~\ref{SameJ} that $j\!\left(  \mathcal{C}_{p,1}\!\left(
			\pm1,d_{0}\right)  \right)  =j\!\left(  \mathcal{C}_{p,2}\!\left(  \pm
			p^{\frac{12}{p-1}},d_{0}\right)  \right)  $. In particular,
			$\mathcal{C}_{p,2}\!\left(  \pm p^{\frac{12}{p-1}},d_{0}\right)  $ is a twist of
			$\mathcal{C}_{p,1}\!\left(  \pm1,d_{0}\right)  $ and thus we only have to
			consider $t_{0}=\pm1$. 
			The table below gives the LMFDB label of $\mathcal{C}_{p,i}\!\left(
			\pm1,d_{0}^{\prime}\right)$, where $d_0'$ gives the quadratic twist of $\mathcal{C}_{p,i}\!\left(
			\pm1,1\right)  $ with smallest conductor and smallest minimal discriminant in absolute value. 
			\[ \scalebox{0.95}{
				\renewcommand{\arraystretch}{1.15}
				\renewcommand{\arraycolsep}{.3cm}
				$\begin{array}
					[c]{cccccccc} \toprule%
					p & d_{0}^{\prime} & \mathcal{C}_{p,1}(1,d_{0}^{\prime}) & \mathcal{C}%
					_{p,2}(1,d_{0}^{\prime}) &  & d_{0}^{\prime} & \mathcal{C}_{p,1}%
					(-1,d_{0}^{\prime}) & \mathcal{C}_{p,2}(-1,d_{0}^{\prime})\\\hline
					2 & -1 & \href{http://www.lmfdb.org/EllipticCurve/Q/4225.h1}{4225.h1} & \href{http://www.lmfdb.org/EllipticCurve/Q/4225.h2}{4225.h2} &  & -3 & \href{http://www.lmfdb.org/EllipticCurve/Q/49.a3}{49.a3} & \href{http://www.lmfdb.org/EllipticCurve/Q/49.a4}{49.a4}\\\hline
					3 & 21 & \href{http://www.lmfdb.org/EllipticCurve/Q/196.a1}{196.a1} & \href{http://www.lmfdb.org/EllipticCurve/Q/196.a2}{196.a2} &  & -39 & \href{http://www.lmfdb.org/EllipticCurve/Q/676.b1}{676.b1} & \href{http://www.lmfdb.org/EllipticCurve/Q/676.b2}{676.b2}\\\hline
					5 & 2 & \href{http://www.lmfdb.org/EllipticCurve/Q/1369.f1}{1369.f1} & \href{http://www.lmfdb.org/EllipticCurve/Q/1369.f2}{1369.f2} &  & -1 & \href{http://www.lmfdb.org/EllipticCurve/Q/43264.f1}{43264.f1} & \href{http://www.lmfdb.org/EllipticCurve/Q/43264.f2}{43264.f2}\\\hline
					7 & 1 & \href{http://www.lmfdb.org/EllipticCurve/Q/3969.f1}{3969.f1} & \href{http://www.lmfdb.org/EllipticCurve/Q/3969.f2}{3969.f2} &  & -1 & \href{http://www.lmfdb.org/EllipticCurve/Q/1369.b1}{1369.b1} & \href{http://www.lmfdb.org/EllipticCurve/Q/1369.b2}{1369.b2}\\\hline
					13 & -2 & \href{http://www.lmfdb.org/EllipticCurve/Q/9025.a1}{9025.a1} & \href{http://www.lmfdb.org/EllipticCurve/Q/9025.a2}{9025.a2} &  & -1 & \href{http://www.lmfdb.org/EllipticCurve/Q/20736.n1}{20736.n1} & \href{http://www.lmfdb.org/EllipticCurve/Q/20736.n2}{20736.n2}\\
					\bottomrule &  &  &  & 
				\end{array} $} \vspace{-1em}
			\]
			The result now follows, since $\mathcal{C}_{p,i}\!\left(  \pm1,d_{0}^{\prime
			}\right)  $ is a twist of $\mathcal{C}_{p,i}\!\left(  \pm1,d_{0}\right)  $,
			and $\mathcal{C}_{p,1}\!\left(  \pm1,d_{0}^{\prime}\right)  $ and
			$\mathcal{C}_{p,2}\!\left(  \pm1,d_{0}^{\prime}\right)  $ are isogenous discriminant
			ideal twins.
		\end{proof}
		
		\begin{proposition}
			\label{CorQuaClass}Let $K$ be an imaginary quadratic field with unit group
			$\left\{  \pm1\right\}  $ and let $p\in\left\{  2,3,5,7,13\right\}  $. Suppose
			that $E_{1}$ and $E_{2}$ are $p$-isogenous discriminant ideal twins over $K$.
			Then one of the following is true:
			
			\begin{enumerate}
				\item If the $j$-invariants of $E_{1}$ and $E_{2}$ are not both identically
				$0$ or $1728$, then $E_{i}$ is a twist of $\mathcal{C}_{p,i}(\pm t_{0},1)$
				where%
				\begin{equation}\label{valt0ImK}
					t_{0}\in\left\{
					\begin{array}
						[c]{cl}%
						\left\{  1,\lambda_{1},\lambda_{2}\right\}   & \text{if }p\mathcal{O}%
						_{K}=\mathfrak{p}_{1}\mathfrak{p}_{2}\text{ with }\mathfrak{p}_{i}^{\frac{12}{p-1}}=\lambda_{i}%
						\mathcal{O}_{K}\text{,}\\
						\left\{  1,\lambda \right\}   & \text{if }p\mathcal{O}_{K}%
						=\mathfrak{p}^{2}\text{ with } \fp^{\frac {12}{p-1}} = \lambda \mathcal{O}_K,\\
						\left\{  1\right\}   & \text{otherwise.}%
					\end{array}
					\right.
				\end{equation}
				\item If the $j$-invariants of $E_{1}$ and $E_{2}$ are both identically $0$ or
				$1728$, then
				\begin{enumerate}
					\item if $j(E_{1})=0$ and $3$ is ramified in $K$, then $E_{i}$
					is a twist of $E_{i,0}(1)$;
					\item if $j(E_{1})=1728$ and $2$ is  ramified in $K$, then $E_{i}$ is a twist of $E_{i,1728}(1)$.
				\end{enumerate}
			\end{enumerate}
			
			Moreover, the converse to $\left(  1\right)  $ and $\left(  2\right)  $ holds
			if $p\neq2$ or if $j(E_1) = j(E_2) \neq 0, 1728$ and the ring of CM endomorphisms is not contained in $K$.
		\end{proposition}
		
		\begin{proof}
			The assumption that $K$ is an imaginary quadratic field with unit group
			$\left\{  \pm1\right\}  $ implies that $i,\zeta_{3}\not \in K$. Consequently,
			if $j(E_{1})=j(E_{2})\in\left\{  0,1728\right\}  $, then the proposition holds
			automatically by Theorem \ref{thmforj-_1728}. So it remains to consider the
			case when the $j$-invariants of $E_{1}$ and $E_{2}$ are not both identically
			$0$ or $1728$. 
			
			To this end, Theorem \ref{PropMinDisc} implies that $E_{i}$ is
			a twist of $\mathcal{C}_{p,i}(t_{0},1)$ where $t_{0}\in\mathcal{O}_{K}$ such
			that
			\begin{equation}
				\nu_{\mathfrak{q}}(t_{0})=\frac{12k}{p-1}\ \text{with\ }0\leq k\leq
				\nu_{\mathfrak{q}}(p)\text{ for each prime ideal }\mathfrak{q}\text{ of
				}K\text{}\label{t0val}%
			\end{equation}
			so long as $j_1 \neq j_2$ or $j_1 = j_2$ and $E_1$ is not isomorphic to $E_2$ over $K$,
			i.e., as long as the CM endomorphisms of $E_1$ and $E_2$ are not contained in $K.$
			We exclude the case $j_1 = j_2$ moving forward. 
			
			Continuing, by Remark~\ref{SameJ}, $\mathcal{C}_{p,1}(\pm1,1)$ is a
			twist of $\mathcal{C}_{p,2}(\pm p^{\frac{12}{p-1}},1)$. Consequently, it
			suffices to find $t_{0}\in\mathcal{O}_{K}\backslash\left\{   \pm p^{\frac
				{12}{p-1}}\right\}  $ which satisfy (\ref{t0val}). 
			We now proceed by cases.
			
			Case 1. Suppose that $p$ is inert in $K$. Then the proof of Proposition
			\ref{classinQ} for $p\in\left\{  2,3,5,7,13\right\}  $ holds when $\mathbb{Q}$
			is replaced by~$K$. Thus, $E_{i}$ is a twist of $\mathcal{C}_{p,i}(\pm1,1)$.
			
			Case 2. Suppose that $p$ splits in $K$ with $p\mathcal{O}_{K}=\mathfrak{p}%
			_{1}\mathfrak{p}_{2}$. 
			Then $\nu_{\fp_i}(p) = 1$ and $\nu_{\fp_i}(t_0) \in \{ 0, \frac{12}{p-1}\}$ for $i = 1, 2$ and $\nu_{\fq}(t_0) = 0$ for all other primes $\fq$.  
			Note that if both $\nu_{\fp_1}(t_0) = \nu_{\fp_2}(t_0) = \frac{12}{p-1}$, then $t_0 = \pm p^{\frac{12}{p-1}}$, which we have already considered.  
			Similarly, if $\nu_{\fp_1}(t_0) = \nu_{\fp_2}(t_0) = 0$, then $t_0 = \pm 1$, which we have also considered.
			Thus the two cases are if $\nu_{\fp_1}(t_0) = \frac{12}{p-1}$ and  $\nu_{\fp_2}(t_0) = 0$ or if $\nu_{\fp_1}(t_0) = 0$ and  $\nu_{\fp_2}(t_0) = \frac{12}{p-1}$.
			Then $\mathcal{O}_K$ has an element $t_0$ with the appropriate valuation if and only $\fp_1^{\frac{12}{p-1}} = \lambda_1\mathcal{O}_K$ is principal. Hence $\fp_2^{\frac{12}{p-1}} = \lambda_2\mathcal{O}_K$ is principal. Due to the Galois action, $\fp_1^{\frac{12}{p-1}}$ is principal if and only if $\fp_2^{\frac{12}{p-1}}$ is principal.
			
			Case 3. Suppose that $p$ ramifies in $K$ with $p\mathcal{O}_{K}=\mathfrak{p}%
			^{2}$. 
			Then $\nup(p) = 2$ and $\nup(t_0) \in \{0, \frac{12}{p-1}, \frac{24}{p-1} \}$.
			For $\nup(t_0) = \frac{24}{p-1}$, $t_0 = \pm p^{\frac{12}{p-1}}$, which has been considered.
			Similarly for $\nup(t_0) = 0$.
			As before, notice we have an element $t_0 \in \mathcal{O}_K$ with $\nup(t_0) = \frac{12}{p-1}$ and $\nu_{\fq}(t_0) = 0$ for all other primes $\fq$ if and only if $\fp^{\frac{12}{p-1}} = \lambda \mathcal{O}_K$ is principal, giving us our result.
			Note that for $p = \{2, 3, 7\}$ the ideal $\fp^{\frac{12}{p-1}}$ is principal automatically as $\frac{12}{p-1}$ is even and thus $\fp^{\frac{12}{p-1}} = p^{6/(p-1)}$.
		\end{proof}
		
		\begin{remark} \label{remark:Qcurves}
			Let $K$ be a number field where $p\OK = \fp^{2e'}$ for $p \in \{2,3,7 \}$.
			The end of the proof of Proposition~\ref{CorQuaClass} demonstrates that some of the $p$-isogenous discriminant ideal twins over $K$ occur as base changes of rational elliptic curves since $\fp^{\frac{12e'}{p-1}}=p^{\frac{6}{p-1}} \OK$.  
			For $p \in \{2, 3, 7\}$ the two curves have the same $j$-invariant, $8000, 54000,$ and $16581375$, $-3375$, respectively, thus they are twists. In fact, these are the curves noted in Theorem \ref{Thm:equaljinv} and they are twists by $-p$, and therefore are $K$-isomorphic over a number field containing $\sqrt{-p}$.
			Their discriminants over $\mathbb{Q}$ and twist parameters are given in Table \ref{liminalDITs}.
			
			\begin{center} 
				{\renewcommand{\arraystretch}{1.2}
					\renewcommand{\arraycolsep}{2cm}
					\label{liminalDITs}
					\begin{tabular}{cccccc} \toprule
						$p$ & $t_0$ & $j$ & $\Delta_{1}^\text{min}$ & $\Delta_{2}^\text{min}$ & CM field \\ \hline
						$2$ & $64$ & $8000$ & $2^9$ & $2^{15}$ & $\mathbb{Q}(-2)$ \\\hline
						$3$ & $27$ & $54000$ & $2^8 3^3$ & $2^8 3^9$ & $\mathbb{Q}(-3)$ \\\hline
						$7$ & $7$ & $16581375$ & $2^{12} 3^6 7^9$ & $2^{12} 3^6 7^3$ & $\mathbb{Q}(-7)$ \\\hline
						$7$ & $-7$ & $-3357$ & $-7^9$ & $-7^3$ & $\mathbb{Q}(-7)$\\ \bottomrule
				\end{tabular}}
			\end{center}
		\end{remark}
		
		
		
			

		From Proposition~\ref{CorQuaClass} and Theorems~\ref{thmforj-_1728} and \ref{Thm:equaljinv}, we obtain an algorithm for $p\in\left\{2,3,5,7,13\right\}  $ to explicitly determine all $p$-isogenous discriminant ideal twins over a fixed imaginary quadratic field with unit group $\left\{\pm1\right\}  $. 
		For instance, given such a field $K$, if $2$ or $3$ are totally ramified, then we have prime isogenous discriminant ideal twins with the same $j$-invariant $0$ or $1728$ by Theorem~\ref{thmforj-_1728}. 
		Once this case is considered, we then apply Proposition~\ref{CorQuaClass} by first factoring the ideals $\{2\mathcal{O}_{K},3\mathcal{O}_{K},5\mathcal{O}_{K},7\mathcal{O}_{K},13\mathcal{O}_{K}\}$. 
		Then, check if the appropriate powers of the prime ideal factors are principal. 
		This allows us to find all possible values of $t_{0}$ that give discriminant ideal twins. 
		Using Theorem~\ref{Thm:equaljinv} we then check that curves with the same $j$-invariant are not $K$-isomorphic. 
		The next corollary illustrates this in the setting of the class number four number field $\mathbb{Q}(\sqrt{-33})$.

		\begin{corollary}\label{corQ33}
			Let $p\in\left\{  2,3,5,7,13\right\}  $ and let $K=\mathbb{Q}\!\left(
			\sqrt{-33}\right)  $. Let $E_{1}$ and $E_{2}$ be $p$-isogenous discriminant
			ideals twins defined over $K$. Then $E_{1}$ and $E_{2}$ are twists of one of
			the following:
			
			\begin{enumerate}
				\item the elliptic curves appearing in Table \ref{ta:clasoverQ};
				
				\item $E_{1,1728}(1)$ and $E_{2,1728}(1)$, respectively, if $p=2$ and
				$j(E_{1})=j(E_{2})=1728$;
				
				\item $E_{1,0}(1)$ and $E_{2,0}(1)$, respectively, if $p=3$ and $j(E_{1}%
				)=j(E_{2})=0$;
				
				\item $\mathcal{C}_{p,1}(t_{0},1)$ and $\mathcal{C}_{p,2}(t_{0},1)$,
				respectively, where $t_{0}$ is one of the elements given in
				Table~\ref{ta:clasoverim33}. The table also lists the  $j$-invariant of the
				elliptic curves and their minimal discriminant ideal.
			\end{enumerate}
		\end{corollary}
		{\renewcommand*{\arraystretch}{1.15}
			\begin{longtable}[c]{cccc}
				\caption{Additional prime isogenous discriminant ideal twins over $\Q(\sqrt{-33})$}\\
				\hline
				$p$ & $t_{0}$ & $j\!\left(  \mathcal{C}_{p,1}(t_{0},1)\right)  $ & $j\!\left(
				\mathcal{C}_{p,2}(t_{0},1)\right)  $ \\ 
				\hline
				\endfirsthead
				\caption[]{\emph{continued}}\\
				\hline
				$p$ & $t_{0}$ & $j\!\left(  \mathcal{C}_{p,1}(t_{0},1)\right)  $ & $j\!\left(
				\mathcal{C}_{p,2}(t_{0},1)\right)  $ \\
				\hline
				\endhead
				\hline
				\multicolumn{4}{r}{\emph{continued on next page}}
				\endfoot
				\hline
				\endlastfoot
				
				$2$ & $64$ & $8000$ & $8000$ \\\hline
				$3$ & $27$ & $54000$ & $54000$  \\\hline
				$7$ & $\sqrt{-33} - 4$ & $53865\sqrt{-33} - 72900$ & $-53865\sqrt{-33} - 72900$ \\\cmidrule(lr){2-4}
				& $\sqrt{-33} + 4$ & $1509417\sqrt{-33} - 1404060$ & $-1509417\sqrt{-33} - 1404060$ 
				\label{ta:clasoverim33}	
		\end{longtable}}
		
		\begin{proof}
			We begin by noting that $p\in \{5,13\}$ is inert in $K$. By Proposition~\ref{CorQuaClass}, the only $p$-isogenous discriminant ideal twins
			are twists of $\mathcal{C}_{p,1}(\pm1,d)$ and $\mathcal{C}_{p,2}(\pm1,d)$.
			For the remaining primes, \eqref{factorsQ33} gives the prime ideal factorization of the ideal $p\mathcal{O}_{K}$.
			\begin{equation}
				\renewcommand{\arraystretch}{1.4}
				\begin{array}
					[c]{ccc}\bottomrule
					2 & 3 & 7 \\\hline
					\left(  2,\sqrt{-33}+1\right)  ^{2}\mathcal{O}_{K} & \left(  3,\sqrt
					{-33}\right)  ^{2}\mathcal{O}_{K}& \left(  7,\sqrt
					{-33}+3\right)  \left(  7,\sqrt{-33}+4\right)  \mathcal{O}_{K}\\\hline
				\end{array}
				\label{factorsQ33}%
			\end{equation}
			In particular, $2$ and $3$ are ramified in $K$. It is then easily checked that $E_{1,0}(1)$ and $E_{2,0}(1)$ are $3$-isogenous discriminant ideal twins over $K$. Similarly, $E_{1,1728}(1)$ and $E_{2,1728}(1)$ are $2$-isogenous discriminant ideal twins over $K$. 
			
			It remains to consider the case where $p\in\left\{  2,3,7\right\}  $ and the $j$-invariants of $E_1$ and $E_2$ are not identically $0$ or $1728$. To this end, observe that $\mathfrak{p}^{\frac{12}{p-1}}$ is a principal ideal where $\mathfrak{p}$ is one of the
			prime ideal factors appearing in (\ref{factorsQ33}). In fact,
			\[
			\begin{array}
				[c]{rclcrcl}
				\left(  2,\sqrt{-33}+1\right)  ^{12}\mathcal{O}_{K} & = & 64\mathcal{O}_{K} &
				\qquad & \left(  7,\sqrt{-33} + 3\right)  ^{2}\mathcal{O}_{K} & = & \left(
				\sqrt{-33}-4\right)  \mathcal{O}_{K}\\
				\left(  3,\sqrt{-33}\right)  ^{6}\mathcal{O}_{K} & = & 27\mathcal{O}_{K} &  &
				\left(  7,\sqrt{-33}+4\right)  ^{2}\mathcal{O}_{K} & = & \left(  \sqrt
				{-33} + 4\right)  \mathcal{O}_{K}
			\end{array}
			.
			\]
			By Proposition \ref{CorQuaClass} it remains to check whether $\mathcal{C}%
			_{p,1}(t_{0},1)$ and $\mathcal{C}_{p,2}(t_{0},1)$ are discriminant ideal twins
			for%
			\begin{equation}
				\left(  p,t_{0}\right)  \in\left\{  \left(  2,\pm64\right)  ,\left(
				3,\pm27\right)  ,\left(  7,\pm\left(  \sqrt{-33}\pm4\right)  \right)
				\right\}  .\label{t0values33}%
			\end{equation}
			We note that $\mathcal{C}_{2,i}(-64,1)$ and $\mathcal{C}_{3,i}(-27,1)$ are
			singular curves. The remaining $\left(  p,t_{0}\right)  $ in (\ref{t0values33}%
			) result in discriminant ideal twins. The corollary now follows since%
			\begin{align*}
				j\!\left(  \mathcal{C}_{7,1}\!\left(  \sqrt{-33}+4,1\right)  \right)    &
				=j\!\left(  \mathcal{C}_{7,2}\!\left(  -\sqrt{-33}+4,1\right)  \right)  ,\\
				j\!\left(  \mathcal{C}_{7,2}\!\left(  \sqrt{-33}+4,1\right)  \right)    &
				=j\!\left(  \mathcal{C}_{7,1}\!\left(  -\sqrt{-33}+4,1\right)  \right)  ,\\
				j\!\left(  \mathcal{C}_{7,1}\!\left(  \sqrt{-33}-4,1\right)  \right)    &
				=j\!\left(  \mathcal{C}_{7,2}\!\left(  -\sqrt{-33}-4,1\right)  \right)  ,\\
				j\!\left(  \mathcal{C}_{7,2}\!\left(  \sqrt{-33}-4,1\right)  \right)    &
				=j\!\left(  \mathcal{C}_{7,1}\!\left(  -\sqrt{-33}-4,1\right)  \right)  .\qedhere
			\end{align*}
		\end{proof}
		
		To conclude our classification of discriminant ideal twins over imaginary
		quadratic fields, it remains to consider $\mathbb{Q}(i)$ and $\mathbb{Q}(\zeta_{3})$, which are the only imaginary quadratic fields with unit group
		not equal to $\left\{  \pm1\right\}  $. This is the setting of our last proposition:
		

		\begin{proposition}\label{PropImKu}
			Let $p\in\left\{  2,3,5,7,13\right\}  $, and let $K$ be either $\mathbb{Q}(i)$ or $\mathbb{Q}(\zeta_{3})$. Let $E_{1}$ and $E_{2}$ be $p$-isogenous discriminant ideal twins defined over $K$. Then $E_{1}$ and $E_{2}$ are twists of the following:
			
			\begin{enumerate}
				\item the elliptic curves appearing in Table \ref{ta:clasoverQ};
				\item $\mathcal{C}_{p,1}(t_{0},1)$ and $\mathcal{C}_{p,2}(t_{0},1)$,
				respectively, where $t_{0}$ is one of the elements given in Table
				\ref{ta:clasoverimK}. The table also lists their $j$-invariants and the ratio
				of their minimal discriminants. In particular, if the ratio is $1$, then
				$\mathcal{C}_{p,1}(t_{0},1)$ and $\mathcal{C}_{p,2}(t_{0},1)$ are discriminant twins.
			\end{enumerate}
		\end{proposition}
		
		{\footnotesize \renewcommand*{\arraystretch}{1.15}
			\begin{longtable}[c]{cccccc}
				\caption{Additional prime isogenous discriminant ideal twins over $\Q(i)$ and $\Q(\zeta_3)$}\\
				\hline
				$K$ & $p$ & $t_{0}$ & $j\!\left(  \mathcal{C}_{p,1}(t_{0},1)\right)  $ & $j\!\left(
				\mathcal{C}_{p,2}(t_{0},1)\right)  $ & $\Delta_{\mathcal{C}_{p,2}}^{\text{min}}/\Delta_{\mathcal{C}_{p,1}}^{\text{min}}$\\ 
				\hline
				\endfirsthead
				\caption[]{\emph{continued}}\\
				\hline
				$K$ & $p$ & $t_{0}$ & $j\!\left(  \mathcal{C}_{p,1}(t_{0},1)\right)  $ & $j\!\left(
				\mathcal{C}_{p,2}(t_{0},1)\right)  $ & $\Delta_{\mathcal{C}_{p,2}}^{\text{min}}/\Delta_{\mathcal{C}_{p,1}}^{\text{min}}$\\
				\hline
				\endhead
				\hline
				\multicolumn{6}{r}{\emph{continued on next page}}
				\endfoot
				\hline
				\endlastfoot
				
				$\mathbb{Q}(i)$ & $2$ &  $i$ & $-196607i - 16776448$ & $-4048i + 767$ & $-i$ \\\cmidrule(lr){3-6}
				&  & $-i$ & $196607i - 16776448$ & $4048i + 767$ & $i$ \\\cmidrule(lr){3-6}
				&  &  $64$ & $8000$ & $8000$ & $-1$ \\\cmidrule(lr){3-6}
				&  & $64i$ & $-3008i - 3328$ & $3008i - 3328$ & $i$ \\\cmidrule(lr){2-6}

				& $3$ & $i$ & $387223660i - 19131120$ & $-460i + 720$ & $-1$ \\\cmidrule(lr){3-6}
				&  & $-i$ & $-387223660i - 19131120$ & $460i + 720$ & $-1$ \\\cmidrule(lr){2-6}

				& $5$ & $i$ & $-29902540624i + 7303907000$ & $1136i - 520$ & $1$ \\\cmidrule(lr){3-6}
				&  & $-i$ & $29902540624i + 7303907000$ & $-1136i - 520$ & $1$ \\\cmidrule(lr){3-6}
				&  & $2i + 11$ & $-732160i + 1009152$ & $732160i + 1009152$ & $1$ \\\cmidrule(lr){3-6}
				&  & $11i + 2$ & $503360i + 114984$ & $-503360i + 114984$ & $1$ \\\cmidrule(lr){2-6}

				& $7$ & $i$ & $587961739625i - 376600566000$ & $-1625i - 6000$ & $-1$ \\\cmidrule(lr){3-6}
				&  & $-i$ & $-587961739625i - 376600566000$ & $1625i - 6000$ & $-1$ \\\cmidrule(lr){2-6}

				& $13$ & $i$ & $142566162324912i + 297218710280952$ & $26352i + 128952$ & $1$ \\\cmidrule(lr){3-6}
				&  & $-i$ & $-142566162324912i + 297218710280952$ & $-26352i + 128952$ & $1$ \\\cmidrule(lr){3-6}
				&  & $2i + 3$ & $4092042240i + 410420736$ & $-4092042240i + 410420736$ & $1$ \\\cmidrule(lr){3-6}
				&  & $3i + 2$ & $-1508569920i - 357843672$ & $1508569920i - 357843672$ & $1$ \\\hline
				
				$\mathbb{Q}(\zeta_3)$ & $2$ & $\zeta_3$ & $16580609\zeta_3 - 195840$ & $-4049\zeta_3 - 3329$ & $\zeta_3^2$ \\\cmidrule(lr){3-6}
				&  & $-\zeta_3$ & $16973823\zeta_3 + 197376$ & $4047\zeta_3 + 4863$ & $-\zeta_3^2$ \\\cmidrule(lr){3-6}
				&  & $\zeta_3^2$ & $-16580609\zeta_3 - 16776449$ & $4049\zeta_3 + 720$ & $\zeta_3$ \\\cmidrule(lr){3-6}
				&  & $-\zeta_3^2$ & $-16973823\zeta_3 - 16776447$ & $-4047\zeta_3 + 816$ & $-\zeta_3$ \\\cmidrule(lr){2-6}
				& $3$ & $\zeta_3$ & $18935047\zeta_3 + 387224415$ & $-495\zeta_3 - 8$ & $\zeta_3$ \\\cmidrule(lr){3-6}
				&  & $-\zeta_3$ & $19328705\zeta_3 - 387222903$ & $423\zeta_3 + 1448$ & $\zeta_3$ \\\cmidrule(lr){3-6}
				&  & $\zeta_3^2$ & $-18935047\zeta_3 + 368289368$ & $495\zeta_3 + 487$ & $\zeta_3^2$ \\\cmidrule(lr){3-6}
				&  & $-\zeta_3^2$ & $-19328705\zeta_3 - 406551608$ & $-423\zeta_3 + 1025$ & $\zeta_3^2$ \\\cmidrule(lr){3-6}
				&  & $27\zeta_3$ & $18981\zeta_3 + 13149$ & $-18981\zeta_3 - 5832$ & $\zeta_3$ \\\cmidrule(lr){3-6}
				&  & $-27\zeta_3$ & $33507\zeta_3 - 11637$ & $-33507\zeta_3 - 45144$ & $\zeta_3$ \\\cmidrule(lr){2-6}
				
				& $5$ & $\zeta_3$ & $23213475001\zeta_3 - 6709180500$ & $179\zeta_3 - 361$ & $\zeta_3^2$ \\\cmidrule(lr){3-6}
				&  & $-\zeta_3$ & $-37821287501\zeta_3 - 7939255500$ & $-2719\zeta_3 - 739$ & $\zeta_3^2$ \\\cmidrule(lr){3-6}
				&  & $\zeta_3^2$ & $-23213475001\zeta_3 - 29922655501$ & $-179\zeta_3 - 540$ & $\zeta_3$ \\\cmidrule(lr){3-6}
				&  & $-\zeta_3^2$ & $37821287501\zeta_3 + 29882032001$ & $2719\zeta_3 + 1980$ & $\zeta_3$ \\\cmidrule(lr){2-6}

				& $7$ & $\zeta_3$ & $-598221714432\zeta_3 - 300947514624$ & $-3072\zeta_3 - 2304$ & $1$ \\\cmidrule(lr){3-6}
				&  & $-\zeta_3$ & $576310764672\zeta_3 + 1054107237168$ & $-10368\zeta_3 - 13392$ & $1$ \\\cmidrule(lr){3-6}
				&  & $\zeta_3^2$ & $598221714432\zeta_3 + 297274199808$ & $3072\zeta_3 + 768$ & $1$ \\\cmidrule(lr){3-6}
				&  & $-\zeta_3^2$ & $-576310764672\zeta_3 + 477796472496$ & $10368\zeta_3 - 3024$ & $1$ \\\cmidrule(lr){3-6}
				&  & $3\zeta_3 + 8$ & $-16982784\zeta_3 - 11720592$ & $16982784\zeta_3 + 5262192$ & $1$ \\\cmidrule(lr){3-6}
				
				
				&  & $-8\zeta_3 -5$ & $-1945944\zeta_3 - 1705077$ & $1945944\zeta_3 + 240867$ & $1$ \\\cmidrule(lr){3-6}
				&  & $8\zeta_3 + 5$ & $-2260440\zeta_3 + 2175405$ & $2260440\zeta_3 + 4435845$ & $1$ \\\cmidrule(lr){2-6}

				& $13$ & $\zeta_3$ & $-37044659937024\zeta_3 + 90029813865984$ & $20736\zeta_3 + 13824$ & $1$ \\\cmidrule(lr){3-6}
				&  & $-\zeta_3$ & $-576859107555264\zeta_3 + 422133441946416$ & $695616\zeta_3 + 296496$ & $1$ \\\cmidrule(lr){3-6}
				&  & $\zeta_3^2$ & $37044659937024\zeta_3 + 127074473803008$ & $-20736\zeta_3 - 6912$ & $1$ \\\cmidrule(lr){3-6}
				&  & $-\zeta_3^2$ & $576859107555264\zeta_3 + 998992549501680$ & $-695616\zeta_3 - 399120$ & $1$ \\\cmidrule(lr){3-6}
				
				&  & $\zeta_3 + 4$ & $-6697506816\zeta_3 - 5868496128$ & $6697506816\zeta_3 + 829010688$ & $1$ \\\cmidrule(lr){3-6}
				&  & $3\zeta_3 + 4$ & $-934053120\zeta_3 + 1982683440$ & $934053120\zeta_3 + 2916736560$ & $1$ \\\cmidrule(lr){3-6}
				&  & $4\zeta_3 + 3$ & $520805376\zeta_3 + 116036928$ & $-520805376\zeta_3 - 404768448$ & $1$ 
				\label{ta:clasoverimK}	
		\end{longtable}}
		
		\begin{proof}
			First, the pairs of elliptic curves appearing in Table \ref{ta:clasoverQ} remain
			discriminant ideal twins after base change to $K$. 
			
			Further, by Theorem \ref{thmforj-_1728}, we may assume that the $j$-invariants of
			$E_{1}$ and $E_{2}$ are not both identically $0$ or $1728$. 
			
			Next, we claim that $E_{1}$ and $E_{2}$ are twists of $\mathcal{C}%
			_{p,1}(ut_{0},1)$ and $\mathcal{C}_{p,2}(ut_{0},1)$, respectively, where
			$u\in\mathcal{O}_{K}^{\times}$ and $t_{0}$ satisfies \ref{valt0ImK}. Indeed,
			while Proposition \ref{CorQuaClass} concerned imaginary quadratic fields with
			unit group $\left\{  \pm1\right\}  $, the proof of part $\left(  1\right)  $
			continues to be true if $\pm t_{0}$ is replaced with $ut_{0}$ for
			$u\in\mathcal{O}_{K}^{\times}$. In particular, $E_{1}$ and $E_{2}$ may be
			twists of $\mathcal{C}_{p,1}(u,1)$ and $\mathcal{C}_{p,2}(u,1)$, respectively,
			where~$u\in\mathcal{O}_{K}^{\times}$. Table \ref{ta:clasoverimK} records these
			pair of elliptic curves when $u\neq\pm1$. Note that the unit groups of $%
			\mathbb{Q}
			(i)$ and $%
			\mathbb{Q}
			(\zeta_{3})$ are $\left\{  \pm1,\pm i\right\}  $ and $\left\{  \pm1,\pm
			\zeta_{3},\pm\zeta_{3}^{2}\right\}  $, respectively. Moreover, as with Remark
			\ref{SameJ}, it is checked that the following equalities hold:
			\[%
			\begin{array}
				[c]{rllcrll}%
				j(\mathcal{C}_{p,1}(\pm1,1)) & = & j(\mathcal{C}_{p,2}(\pm p^{\frac{12}{p-1}%
				},1)) & \qquad & j(\mathcal{C}_{p,1}(\pm i,1)) & = & j(\mathcal{C}_{p,2}(\mp
				ip^{\frac{12}{p-1}},1)),\\
				j(\mathcal{C}_{p,1}(\pm\zeta_{3},1)) & = & j(\mathcal{C}_{p,2}(\pm\zeta
				_{3}^{2}p^{\frac{12}{p-1}},1)) & \qquad & j(\mathcal{C}_{p,1}(\pm\zeta_{3}%
				^{2},1)) & = & j(\mathcal{C}_{p,2}(\pm\zeta_{3}p^{\frac{12}{p-1}},1)).
			\end{array}
			\]
			This shows the claim, and as a consequence, we have that if $E_{1}$ and
			$E_{2}$ are not twists of $\mathcal{C}_{p,1}(z,1)$ and $\mathcal{C}%
			_{p,2}(z,1)$, respectively, where $z\in\mathcal{O}_{K}^{\times}$ or
			$\left\vert z\right\vert =p^{\frac{12}{p-1}}$, then $E_{1}$ and $E_{2}$ are
			twists of $\mathcal{C}_{p,1}(t_{0},1)$ and $\mathcal{C}_{p,2}(t_{0},1)$,
			respectively, where%
			\begin{equation}
				t_{0}\in\left\{
				\begin{array}
					[c]{cl}%
					\left\{  \lambda_{1},\lambda_{2}\right\}   & p\mathcal{O}_{K}=\mathfrak{p}%
					_{1}\mathfrak{p}_{2}\ \text{with }\mathfrak{p}^{\frac{12}{p-1}}=\lambda
					_{i}\mathcal{O}_{K},\\
					\left\{  \lambda\right\}   & p\mathcal{O}_{K}=\mathfrak{p}^{2}\text{ with
					}\mathfrak{p}^{\frac{12}{p-1}}=\lambda\mathcal{O}_{K}.
				\end{array}
				\right.  \label{t0valsImKsp}%
			\end{equation}
			In particular, we only have to consider the cases when $p\mathcal{O}_{K}$
			ramifies or splits completely. We now proceed by cases.
			
			Case 1. Let $K=%
			\mathbb{Q}
			(i)$. The prime ideal factorization of $p\mathcal{O}_{K}$ is given in
			(\ref{Qifactors}).%
			\begin{equation}
				\renewcommand{\arraystretch}{1.4}
				\begin{array}
					[c]{ccccc}\bottomrule
					2 & 3 & 5 & 7 & 13\\\hline
					(1+i)^{2}\mathcal{O}_{K} & 3\mathcal{O}_{K} & (2+i)(2-i)\mathcal{O}_{K} &
					7\mathcal{O}_{K} & (3+2i)(3-2i)\mathcal{O}_{K}\\\bottomrule
				\end{array}
				\label{Qifactors}%
			\end{equation}
			By the above, new discriminant ideal twins will only occur for $p\in\left\{
			2,5,13\right\}  $. For these $p$, let $\mathfrak{p}$ be one of the prime ideal
			factors appearing in (\ref{Qifactors}). Then we have that $\mathfrak{p}%
			^{\frac{12}{p-1}}=\mathfrak{p}$ for $p=13$, and for the remaining cases we
			have that $\mathfrak{p}^{\frac{12}{p-1}}$ is:%
			\[
			\left(  1+i\right)  ^{12}\mathcal{O}_{K}=64\mathcal{O}_{K},\qquad\left(
			2+i\right)  ^{3}\mathcal{O}_{K}=\left(  11i+2\right)  \mathcal{O}_{K}%
			,\qquad\left(  2-i\right)  ^{3}\mathcal{O}_{K}=\left(  11i-2\right)
			\mathcal{O}_{K}.
			\]
			Consequently, we have that if $t_{0}$ satisfies (\ref{t0valsImKsp}), then%
			\[
			t_{0}\in\left\{
			\begin{array}
				[c]{ll}%
				\left\{  64u\mid u\in\{\pm1,\pm i\}\right\}   & \text{if }p=2,\\
				\left\{  \left(  11i\pm2\right)  u\mid u\in\{\pm1,\pm i\}\right\}   & \text{if
				}p=5,\\
				\left\{  \left(  3\pm2i\right)  u\mid u\in\{\pm1,\pm i\}\right\}   & \text{if
				}p=13\text{.}%
			\end{array}
			\right.
			\]
			For these $t_{0}$, we note that $\mathcal{C}_{p,i}(t_{0},1)$ is singular if
			$\left(  p,t_{0}\right)  \in\left\{  \left(  2,-64\right)  ,\left(
			5,\pm2i-11\right)  ,\left(  13,\pm2i-3\right)  \right\}  $. We also note that
			$\mathcal{C}_{p,1}(t_{0},1)$ and $\mathcal{C}_{p,2}(t_{0},1)$ are
			$K$-isomorphic if $\left(  p,t_{0}\right)  \in\left\{  \left(  5,\pm
			11i-2\right)  ,\left(  13,\pm3i-2\right)  \right\}  $. It is then verified
			that $\mathcal{C}_{p,1}(t_{0},1)$ and $\mathcal{C}_{p,2}(t_{0},1)$ are
			discriminant ideal twins for the remaining $t_{0}$. Next, let $\overline
			{t_{0}}$ denote the complex conjugate of $t_{0}$. The proposition now follows
			for this case since $j\!\left(  \mathcal{C}_{p,1}(t_{0},1)\right)  =j\!\left(
			\mathcal{C}_{p,2}(\overline{t_{0}},1)\right)  $ and $j\!\left(  \mathcal{C}%
			_{p,2}(t_{0},1)\right)  =j\!\left(  \mathcal{C}_{p,1}(\overline{t_{0}%
			},1)\right)  $ for%
			\[
			\left(  p,t_{0}\right)  \in\left\{  \left(  2,64i\right)  ,\left(
			5,11i+2\right)  ,\left(  5,2i+11\right)  ,\left(  13,2i+3\right)  ,\left(
			13,3i+2\right)  \right\}  .
			\]

			Case 2. Let $K=%
			\mathbb{Q}
			(\zeta_{3})$. The prime ideal factorization of $p\mathcal{O}_{K}$ is given in
			(\ref{Qz3factors}).%
			\begin{equation}
				\renewcommand{\arraystretch}{1.4}
				\begin{array}
					[c]{ccccc}\bottomrule
					2 & 3 & 5 & 7 & 13\\\hline
					2\mathcal{O}_{K} & (\zeta_{3}+2)^{2}\mathcal{O}_{K} & 5\mathcal{O}_{K} &
					(\zeta_{3}+3)(\zeta_{3}-2)\mathcal{O}_{K} & (\zeta_{3}+4)(3\zeta
					_{3}+4)\mathcal{O}_{K}\\\bottomrule
				\end{array}
				\label{Qz3factors}%
			\end{equation}
			By the discussion preceding Case 1, we have that new discriminant ideal twins
			will only occur for $p\in\left\{  3,7,13\right\}  $. For these $p$, let
			$\mathfrak{p}$ be one of the prime ideal factors appearing in
			(\ref{Qz3factors}). Then we have that $\mathfrak{p}^{\frac{12}{p-1}%
			}=\mathfrak{p}$ for $p=13$, and for the remaining cases we have that
			$\mathfrak{p}^{\frac{12}{p-1}}$ is:%
			\[
			(\zeta_{3}+2)^{6}=27\mathcal{O}_{K},\qquad(\zeta_{3}+3)^{2}\mathcal{O}%
			_{K}=(8\zeta_{3}+3)\mathcal{O}_{K},\qquad(\zeta_{3}-2)^{2}\mathcal{O}%
			_{K}=(8\zeta_{3}+5)\mathcal{O}_{K}.
			\]
			Consequently, we have that if $t_{0}$ satisfies (\ref{t0valsImKsp}), then%
			\begin{equation}
				t_{0}\in\left\{
				\begin{array}
					[c]{ll}%
					\left\{  27u\mid u\in\{\pm1,\pm\zeta_{3},\pm\zeta_{3}^{2}\}\right\}   &
					\text{if }p=3,\\
					\left\{  (8\zeta_{3}+3)u,(8\zeta_{3}+5)u\mid u\in\{\pm1,\pm\zeta_{3},\pm
					\zeta_{3}^{2}\}\right\}   & \text{if }p=7,\\
					\{\zeta_{3}+4)u,(3\zeta_{3}+4)u\mid u\in\{\pm1,\pm\zeta_{3},\pm\zeta_{3}%
					^{2}\} & \text{if }p=13\text{.}%
				\end{array}
				\right.  \label{t0ze3vals}%
			\end{equation}
			For these $t_{0}$, we note that $\mathcal{C}_{p,i}(t_{0},1)$ is singular if
			\[
			\left(  p,t_{0}\right)  \in\left\{  \left(  3,-27\right)  ,\left(
			7,3\zeta_{3}-5\right)  ,\left(  7,-3\zeta_{3}-8\right)  ,\left(
			13,3\zeta_{3}-1\right)  ,\left(  13,-3\zeta_{3}-4\right)  \right\}  .
			\]
			In addition, $\mathcal{C}_{p,1}(t_{0},1)$ and $\mathcal{C}_{p,2}(t_{0},1)$ are
			$K$-isomorphic if%
			\[
			\left(  p,t_{0}\right)  \in\left\{  \left(  3,27\right)  ,\left(  7,\pm
			\sigma(5\zeta_{3}-3)\right)  ,\left(  13,\sigma(\zeta_{3}-3)\right)  ,\left(
			13,\sigma(4\zeta_{3}+1)\right)  \mid\sigma\in\operatorname*{Gal}(K/\mathbb{Q})\right\}  .
			\]
			It is then verified that for the remaining $t_{0}$, $\mathcal{C}_{p,1}(t_{0},1)$ and $\mathcal{C}_{p,2}(t_{0},1)$ are discriminant ideal
			twins. This concludes the proof since $j\!\left(  \mathcal{C}_{p,1}%
			(t_{0},1)\right)  =j\!\left(  \mathcal{C}_{p,2}(\overline{t_{0}},1)\right)  $
			and $j\!\left(  \mathcal{C}_{p,2}(t_{0},1)\right)  =j\!\left(  \mathcal{C}%
			_{p,1}(\overline{t_{0}},1)\right)  $ for%
			\[
			\left(  p,t_{0}\right)  \in\left\{  \left(  3,\pm27\zeta_{3}\right)  ,\left(
			7,3\zeta_{3}+8\right)  ,\left(  7,\pm(8\zeta_{3}+5)\right)  ,\left(
			13,\zeta_{3}+4\right)  ,\left(  13,3\zeta_{3}+4\right)  ,\left(  13,4\zeta
			_{3}+3\right)  \right\}  . \qedhere
			\]
		\end{proof}

		\bibliographystyle{amsplain}
		\bibliography{Disc_twins}
		
	\end{document}